\documentclass[a4paper]{article}
\usepackage{amsmath,amsfonts,amssymb,amsthm}
\usepackage{graphicx}
\usepackage{tikz}
\usetikzlibrary{cd,positioning}
\usepackage{eucal}
\usepackage{makeidx}
\usepackage{stackrel}
\usepackage[style=ext-alphabetic,doi=false,url=false,isbn=false,giveninits=true,uniquename=init,
articlein=false, backref=false]{biblatex}\usepackage[dvipdfmx, colorlinks=true, allcolors= blue, bookmarks=true]{hyperref}\usepackage{cleveref}
\definecolor{blue}{HTML}{065FD4}
\renewbibmacro{in:}{}

\newbibmacro{string+doiurl}[1]{ \iffieldundef{doi}
 {\iffieldundef{url}
 {#1}
 {\href{\thefield{url}}{#1}}}
 {\href{https://doi.org/\thefield{doi}}{#1}}}
\makeatletter
\def\blx@driver#1{ \ifcsdef{blx@bbx@#1}
 {\usebibmacro{string+doiurl}{\csuse{blx@bbx@#1}}}}
\makeatother
\DeclareSourcemap{
 \maps[datatype=bibtex, overwrite]{
 \map{
 \step[fieldset=editor, null]
 }
 }
}
\newtheoremstyle{mystyle} {} {} {\normalfont} {} {\bfseries} {} { } {}\theoremstyle{definition}
\newtheorem{definition}{Definition}[section]
\newtheorem{example}[definition]{Example}
\theoremstyle{plain}
\newtheorem{proposition}[definition]{Proposition}
\newtheorem{lemma}[definition]{Lemma}
\newtheorem{theorem}[definition]{Theorem}

\theoremstyle{definition}
\newtheorem{remark}[definition]{Remark}
\newtheorem*{remark*}{Remark}

\crefname{section}{Section}{Sections}
\crefname{definition}{Definition}{Definitions}
\crefname{proposition}{Proposition}{Propositions}
\crefname{lemma}{Lemma}{Lemmas}
\crefname{theorem}{Theorem}{Theorems}
\crefname{corollary}{Corollary}{Corollaries}
\crefname{remark}{Remark}{Remarks}
\crefname{example}{Example}{Examples}
\crefname{conjecture}{Conjecture}{Conjectures}
\crefname{notation}{Notation}{Notations}
\newcommand{\coker}{\mathop{\mathrm{coker}}\nolimits}
\newcommand\Hom{\mathop{\mathrm{Hom}}\nolimits}
\newcommand\stHom{\mathop{\underline{\mathrm{Hom}}}\nolimits}
\newcommand\prHom{\mathop{\mathrm{Hom}^{\mathrm{pr}}}\nolimits}
\newcommand\End{\mathop{\mathrm{End}}\nolimits}
\newcommand\soc{\mathop{\mathrm{soc}}\nolimits}
\newcommand\rad{\mathop{\mathrm{rad}}\nolimits}

\newcommand\op{^\mathrm{op}}
\newcommand{\Z}{\mathbb{Z}}
\newcommand\derb{\mathop{D^b}}
\newcommand\modcat{\mathop{\mathrm{mod}}\nolimits}
\newcommand\stmodcat{\mathop{\underline{\mathrm{mod}}}\nolimits}
\newcommand\homob{\mathop{K^b}}
\newcommand\projcat{\mathop{\mathrm{proj}}}
\newcommand\addcat{\mathop{\mathrm{add}}}
\newcommand\pathtohom[1]{\,{\vphantom{#1}}^{\sharp}\:\!\! #1}
\newcommand\image{\mathop{\mathrm{Im}}\nolimits}
\renewcommand\Im{\mathop{\mathrm{Im}}\nolimits}
\newcommand\Ker{\mathop{\mathrm{Ker}}\nolimits}
\newcommand\Coker{\mathop{\mathrm{Coker}}\nolimits}
\newcommand\syzygy[1][]{\mathop{\Omega^{\scriptstyle #1}}}\newcommand{\crefenum}[2]{ \namecref{#1}~\hyperref[#2]{\labelcref*{#1}~\ref*{#2}}}

\newcommand{\myfbox}[1]{\kern1pt\fbox{#1}\kern1pt}

\title{Two-sided tilting complexes for generalized Brauer tree algebras}
\author{Shuji Fujino \and Yuta Kozakai \and Kohei Takamura}
\date{}
\begin{document}
\maketitle
\begin{abstract}
  We explicitly construct two-sided tilting complexes corresponding to Membrillo-Hern\'{a}ndez's tree-to-star tilting complexes for generalized Brauer tree algebras.
\end{abstract}
\section{Introduction}
In the article, we deal with the algebras called generalized Brauer tree algebras, which is symmetric and of tame representation type. These algebras generalize Brauer tree algebras and are Brauer graph algebras associated to tree-shaped Brauer graphs.

Let $k$ be an algebraically closed field. Let $A$ be a generalized Brauer tree algebra over $k$ associated to a generalized Brauer tree $\mathcal T=(V,E,m,\rho)$, where $V$ denotes the set of vertices, $E$ the set of edges, $m$ the multiplicities and $\rho$ cyclic orders. Fix a vertex $v_0\in V$.

If the multiplicity at any vertex except for $v_0$ is equal to $1$, the algebra $A$ is known as a Brauer tree algebra with an exceptional vertex $v_0$, which is of finite representation type.
Brauer tree algebras are appeared in the modular representation theory of finite groups as the blocks of group algebras with cyclic defect groups. There are many studies on these algebras. For the Brauer tree algebra $A$, Rickard showed that there is a tilting complex $T_0$ of $A$-modules whose endomorphism ring is isomorphic to a Brauer tree algebra associated to a star-shaped Brauer tree which has the same multiplicity at the center of the star as that of $v_0$ and has the same number of edges as those of $\mathcal T$ (\cite[Theorem 4.2]{MR1027750}). By using this fact, he proved that any two Brauer tree algebras are derived equivalent if and only if they have the same number of edges and have the same multiplicity at their exceptional vertices.

In general, for a finite dimensional algebra $\Gamma$ and a tilting complex $P^{\bullet}$ of $\Gamma$-modules, there exists a derived equivalence $\derb(\Lambda) \to \derb(\Gamma)$, where we put $\Lambda:= \End_{\homob(\projcat\Lambda)}(P^{\bullet})$, which is known as Rickard's theorem (\cite{Ricardtheorem}). The functor constructed in the article is designed for theoretical purposes and may be difficult for practical calculations.
Then, Keller and Rickard showed that in the same setting, there exists a two-sided tilting complex $C$ of $(\Lambda, \Gamma)$-bimodules whose restriction to $\Gamma$ is isomorphic to $P^{\bullet}$ in the derived category of $\Gamma$-modules (\cite{Keller1993,Rickard1991}). Moreover, the complex $C$ gives a left derived functor ${-} \otimes_{\Lambda}^{\mathbb L} C: \derb{(\Lambda)}\to\derb{(\Gamma)}$ of ${-} \otimes_\Lambda C$. The constructions of such two-sided tilting complexes may be difficult to calculate images of complexes through the functor.
Therefore, Kozakai-Kunugi gave the explicit description of a two-sided tilting complex for a Rickard's tree-to-star complex $T_0$ of a Brauer tree algebra $A$ by using a bimodule giving a stable equivalence of Morita type between $A$ and the endomorphism ring of the Rickard's tree-to-star complex(\cite{MR3895203}).

On the other hand, in the situation that $A$ is a generalized Brauer tree algebra not necessarily a Brauer tree algebra, which is of tame representation type but not necessarily of finite representation type, Membrillo-Hern\'{a}ndez showed that there is a tilting complex $T$ of $A$-modules whose endomorphism ring which we denote by $B$ is isomorphic to a generalized Brauer tree algebra associated to a star-shaped generalized Brauer tree with the same number of edges and with the same multiplicities in some order determined by Green walk on the edges (\cite[Theorem 7.4]{MEMBRILLOHERNANDEZ1997231}). By using this fact, he proved that any two generalized Brauer tree algebras are derived equivalent if and only if they have the same number of edges and the set of multiplicities (\cite[Theorem 9.7]{MEMBRILLOHERNANDEZ1997231}). The fact is a generalization of the one made in \cite{MR1027750}.

By analogy, we naturally expect that the method given in \cite{MR3895203} can be applied to construct a two-sided tilting complex corresponding to the Membrillo-Hern\'{a}ndez's tree-to-star tilting complex. However, the proofs Kozakai-Kunugi gave depends on the fact that given a simple $B$-module $V$, any simple $B$-module is isomorphic to $\syzygy[2\ell] V$ for some $\ell\in\mathbb Z$. This does not hold for generalized Brauer tree algebras in general.
Therefore, in this paper, we aim to explicitly construct a two-sided tilting complex for a Membrillo-Hern\'{a}ndez's tree-to-star complex $T$ of a generalized Brauer tree algebra $A$.
After all, the essential parts of proofs in this article are done by perverse equivalences which is introduced in \cite{Perverse}. We prove that a Membrillo-Hern\'{a}ndez's tilting complex $T$ induces a perverse equivalence. Moreover we calculated images of simple modules through the functor.

Now we briefly describe the two-sided tilting complex for a generalized Brauer tree algebra. Let $M$ denote some indecomposable $B^{\op}\otimes_k A$-module inducing a stable equivalence of Morita type.
By taking a minimal projective resolution of $M$ of $B^{\op}\otimes_k A$-module and deleting some direct summands of each term, we get a two-sided tilting complex $C$ of $B^{\op}\otimes_k A$-modules (see \cref{Cistwo-sided}, and \cref{s:main} for construction). Moreover, the restriction of $C$ to $A$ is isomorphic to $T$ in the derived category of $A$-modules (see \cref{thm:resisC}). The functor $G:={-}\otimes_B^{\mathbb L} C$ gives a derived equivalence $\derb(B)\to \derb(A)$. Furthermore, we have $G={-}\otimes_B C$ since $C$ is projective on either side by the construction of $C$. The functor $G$ is explicit and easy for us to deal with.

The organization of the article is as follows.

In \cref{s:notation}, we give the notation for the article.
In \cref{s:tilting}, we recall the basic facts for tilting complexes. The facts for stable equivalences of Morita type are collected in \cref{s:stable}. In \cref{s:perverse}, we introduce the notion of perverse equivalences for symmetric algebras. \cref{s:graphs,s:gbta,s:modules} are devoted to introduce the terminology for generalized Brauer tree and describe how to construct generalized Brauer tree algebras from generalized Brauer tree and their representation.
In \cref{s:main}, we construct the complex of $(B,A)$-bimodules and we give the theorems and proofs for the complex.
Finally, in \cref{s:examples} we give an example of the theorems in \cref{s:main}.

\section{Notation}\label{s:notation}
Throughout this paper, $\Gamma$ and $\Lambda$ mean finite dimensional indecomposable symmetric algebras over an algebraically closed field $k$. We denote by $\Gamma^{\op}$ the opposite algebra of $\Gamma$. We ensure that $\Gamma$-modules mean finitely generated right $\Gamma$-modules. We identify $(\Lambda,\Gamma)$-bimodules with $\Lambda^{\op}\otimes_k \Gamma$-modules.
Let $f:U\to V$ be a $\Gamma$-homomorphism. We denote the kernel of $f$ by $\Ker f$, the cokernel of $f$ by $\Coker f$ and the image of $f$ by $\Im f$. We denote the canonical monomorphism $\Ker f\to U$ by $\ker f$ and the canonical epimorphism $V\to \Coker f$ by $\coker f$.
Given a $\Gamma$-module $U$, we denote by $\rad(U)$ the radical of $U$, by $\soc(U)$ the socle of $U$, by $P(U)$ a projective cover of $U$, by $I(U)$ an injective hull of $U$, by $\syzygy U$ the kernel of a projective cover of $U$, and by $\syzygy[-1] U$ the cokernel of an injective hull of $U$. We define inductively $\syzygy[n+1] U=\syzygy (\syzygy[n] U)$ and $\syzygy[-n-1] U=\syzygy[-1](\syzygy[-n] U)$ for each integer $n\ge 1$ and a $\Gamma$-module $U$.
We denote by $U^{\ast}:=\Hom_{\Gamma}(U,\Gamma)$ the dual module of a $\Gamma$-module $U$. This is $\Gamma^{\op}$-module. We note that $V^{\ast}\otimes_k U$ is $(\Lambda,\Gamma)$-bimodule for a $\Gamma$-module $U$ and $\Lambda$-module $V$.
Let $\{C^n\}_{n\in\mathbb Z}$ be a set of $\Gamma$-modules indexed by $\Z$. Let $d^n$ be a $\Gamma$-homomorphism from $C^n$ to $C^{n+1}$, which satisfies $d^{n+1}\circ d^n=0$ for each integer $n$. Then $C=(C^n, d^n)_{n\in \Z}$ is a {complex} of $\Gamma$-modules. We say that a complex $C=(C^n, d^n)_{n\in \Z}$ is {bounded} if the set of $n$ such that $C^n$ is nonzero is finite.
Given complexes of $\Gamma$-modules $C=(C^n,d^n)$ and $D=(D^n, {d'}^n)$, we say that $f=(f^n)_{n\in \Z}$ is a {homomorphism of complexes} of $\Gamma$-modules from $C$ to $D$ if $f^n$ is a $\Gamma$-homomorphism from $C^n$ to $D^n$ satisfying $f^{n+1}\circ d^{n}={d'}^n\circ f^n$ for each integer $n$. We denote by $C[n]$ the $n$-shifted complexes of $C$ and by $H^n(C)$ the $n$th cohomology of $C$ for each $n$.

We denote the category of finitely generated right $\Gamma$-modules by $\modcat \Gamma$, the full subcategory of $\modcat \Gamma$ whose objects are finitely generated right projective $\Gamma$-modules by $\projcat \Gamma$,
the bounded homotopy category of $\projcat \Gamma$ by $\homob(\projcat \Gamma)$, the bounded derived category of $\modcat \Gamma$ by $\derb (\Gamma)$, and the stable category of $\modcat \Gamma$ by $\stmodcat{\Gamma}$.

Given a set $X$, we denote by $\delta_{ij}$ the Kronecker delta for $i, j \in X$ given by \[\delta_{ij}=\begin{cases}
    1 & \text{if $i=j$,}     \\
    0 & \text{if $i\neq j$}.
  \end{cases}\]

\section{Tilting complexes}\label{s:tilting}
Tilting complexes are important to consider derived equivalences. We describe them in this section.
Let $k$ be an algebraically closed field and $\Gamma$ and $\Lambda$ two finite dimensional indecomposable symmetric $k$-algebras. We say that $\Gamma$ and $\Lambda$ are derived equivalent if $\derb(\Gamma)$ and $\derb(\Lambda)$ are equivalent as triangulated categories.
\begin{definition}
  We call a bounded complex $T$ of projective $\Gamma$-modules a \emph{tilting complex} if the following conditions are satisfied:
  \begin{itemize}
    \item $\Hom_{\derb(\Gamma)}(T, T[n])=0$ for any nonzero integer $n$.
    \item The subcategory $\addcat(T)$ generates $\derb(\Gamma)$ as a triangulated category.
  \end{itemize}
\end{definition}
The second condition means that $\Gamma$ is obtained by applying a finite sequence of operations, including taking direct sums, direct summands, mapping cones, and shifts of $T$. We recall the definition of two-sided tilting complexes.
\begin{definition}[{\cite[Definition 3.4.]{Rickard1991}}]
  We call a bounded complex $C$ of $\Lambda^{\op}\otimes_k \Gamma$-modules a \emph{two-sided tilting complex} if
  ${-} \otimes_\Lambda^{\mathbb L} C$
  is an exact functor $\derb(\Lambda)\to \derb(\Gamma)$ which induces an equivalence.
\end{definition}
We note that ${-} \otimes_\Lambda^{\mathbb L} C$ is the left derived functor for ${-} \otimes_\Lambda C$. The following proposition holds.
\begin{proposition}[{\cite[Sec.~4]{Rickard1991}}]\label{omittingL}
  Let $C$ be a two-sided tilting complex of $\Lambda^{\op}\otimes_k \Gamma$-modules such that projective as $\Lambda^{\op}$-modules and $\Gamma$-modules, then ${-} \otimes_\Lambda^{\mathbb L} C$ is equivalent to ${-} \otimes_\Lambda C$ as a functor.
\end{proposition}

The following proposition ties tilting complexes, two-sided tilting complexes and derived equivalences.

\begin{proposition}[{\cite[Theorem 4.1.]{Brou1994}, \cite[Theorem 3.3]{Rickard1991}}]\label{prop:Rickard} The following conditions are equivalent:
  \begin{itemize}
    \item The $k$-algebra $\Gamma$ is derived equivalent to $\Lambda$.
    \item There exists a tilting complex $T$ of $\Gamma$-modules such that $\End_{\derb(\Gamma)}(T)$ is isomorphic to $\Lambda$ as a $k$-algebra.
    \item There exists a two-sided tilting complex of $\Lambda^{\op}\otimes_k \Gamma$-modules.
  \end{itemize}
\end{proposition}

\section{Stable equivalences of Morita type}\label{s:stable}
In this section, we outlook the basic results on stable equivalences, which are weaker equivalences than derived equivalences.
Let $\Gamma$ and $\Lambda$ be two finite dimensional indecomposable symmetric $k$-algebras. For $\Gamma$-modules $U$ and $V$, we denote the $k$-linear space of all homomorphisms from $U$ to $V$ which pass through projective modules by $\prHom(U,V)$. The stable category of $\Gamma$-module denoted by $\stmodcat{\Gamma}$ is defined as follows:
\begin{itemize}
  \item The objects are the same as as those of $\Gamma$-module category.
  \item For $\Gamma$-modules $U$ and $V$, the set of arrows from $U$ to $V$ is $\Hom(U,V)/\prHom(U,V)$. We denote this by $\stHom(U,V)$.
\end{itemize} In addition, the category $\stmodcat{\Gamma}$ is a triangulated category with the auto functor $\syzygy[-1]$.

\begin{definition}[{\cite[\S5]{Brou1994}}]
  We say that $\Gamma$ and $\Lambda$ are \emph{stably equivalent of Morita type} if there exist a $(\Lambda, \Gamma)$-bimodule $M$ and a $(\Gamma, \Lambda)$-bimodule $N$ satisfying the following conditions.
  \begin{itemize}
    \item The bimodules $M$ and $N$ are projective as left modules and right modules.
    \item $N\otimes_\Lambda M\cong \Gamma\oplus P$ as $(\Gamma,\Gamma)$-bimodules for some projective $(\Gamma,\Gamma)$-bimodule $P$.
    \item $M\otimes_\Gamma N\cong \Lambda\oplus Q$ as $(\Lambda,\Lambda)$-bimodules for some projective $(\Lambda,\Lambda)$-bimodule $Q$.
  \end{itemize}
  We say that $M$ \emph{induces a stable equivalence of Morita type}.
\end{definition}

We remark that the above $(\Lambda, \Gamma)$-bimodule $M$ induces a functor ${-}\otimes_\Lambda M: \modcat{\Lambda} \to \modcat{\Gamma}$, which induces a stable equivalence $\stmodcat{\Lambda} \to \stmodcat{\Gamma}$.

We have the following proposition
noting that by \cite[Theorem 2.1]{MR1027750},
the quotient category $\derb (\Gamma)/\homob (\projcat \Gamma)$ is equivalent to the stable module category of $\Gamma$ as a triangulated category,

\begin{proposition}[{\cite[Corrolary 5.5]{Rickard1991}}]\label{derivedandstable}
  If $F:\derb(\Gamma)\to \derb(\Lambda)$ is a derived equivalence, then there is an indecomposable $(\Lambda,\Gamma)$-bimodule $M$ inducing a stable equivalence of Morita type which commutes with the following diagram up to equivalences.
  \[\begin{tikzcd}[]
      \derb(\Lambda)\dar[""]\rar["F^{-1}"]&\derb(\Gamma)\dar[""]
      \\
      \derb(\Lambda)/\homob(\projcat\Lambda) \dar[phantom,"\rotatebox{-90}{$\cong$}"] & \derb(\Gamma)/\homob(\projcat\Gamma) \dar[phantom,"\rotatebox{-90}{$\cong$}"]\\
      \stmodcat{\Lambda}\rar["-\otimes_{\Lambda} M"]&\stmodcat{\Gamma}
    \end{tikzcd}. \]

\end{proposition}
We remark that the indecomposability of $M$ in the above proposition is
followed by \cite[Proposition 2.4]{Linckelmann1996} and \cite[Proposition 2.2]{DUGAS2007421}. Let $\mathcal S'$ denote a complete set of representatives of isomorphism classes of simple $\Lambda$-modules.

\begin{proposition}[{\cite[Lemma 2]{projcover}}]\label{projcoverRouquier}
  Let $M$ be a $\Lambda^{\op}\otimes_k \Gamma$-module, projective as a $\Gamma$-module and as a
  $\Lambda^{\op}$-module. A projective cover of $M$ is isomorphic to \[\bigoplus_{V\in \mathcal S'} P(V)^{\ast} \otimes_k P(V\otimes_\Lambda M). \]
\end{proposition}

Let $M$ be an indecomposable $\Lambda^{\op}\otimes_k \Gamma$-module inducing a stable equivalence of Morita type between $\Lambda$ and $\Gamma$.
Let $P=(P^t, d_M^t)_{t\in \Z}$ be a projective resolution of the $\Lambda^{\op}\otimes_k \Gamma$-module $M$. Since $\syzygy[\ell] M$ gives a stable equivalence of Morita type for $\ell \ge 1$ (see \cite[Comparison 2.3.5]{rouquier2001a} and \cite[Lemma 4.2]{MR3895203}), the definition of a minimal projective resolution and \cref{projcoverRouquier} give the following proposition.
\begin{proposition}\label{projresol}
  The $(\Lambda,\Gamma)$-bimodule $P^{-t}$ is isomorphic to $M$ for $t=0$, to $0$ for all $t<0$, and
  to \[\bigoplus_{V\in \mathcal S'} P(V)^{\ast}\otimes_k P(V\otimes_{\Lambda} \syzygy[t-1] M)\] for all $t>0$.
\end{proposition}
The following proposition is useful to consider a minimal projective resolution of the module $M$.
\begin{proposition}[{\cite[Lemma 4.3]{MR3895203}}]\label{prop:hellermove}
  We have $V \otimes_\Lambda \syzygy[\ell] M \cong \syzygy[\ell] (V \otimes_\Lambda M )$ for any simple $\Lambda$-module $V$ and $\ell\ge 0$.
\end{proposition}

The following proposition holds similarly to those of propositions in \cite{Rouquier1998} and \cite{MR3895203}. This is important to show that some complexes are two-sided tilting complexes.
\begin{proposition}[see {\cite[Sec.~10.3.4]{Rouquier1998} and \cite[Propositions 4.4 and 5.1]{MR3895203}}]\label{prop:two-sided} Let $C=(C^{\ell}, d_C^{\ell})$ be a bounded complex of $\Lambda^{\op}\otimes_k \Gamma$-modules such that $C^t=0$ for all $t>0$, $C^{0} = M$ and $C^t$ is projective for all $t<0$. If \[\Hom_{\derb(\Gamma)}(V\otimes_\Lambda C, W\otimes_\Lambda C[-n]) \cong \delta_{VW}\delta_{n0}k\] for $V,W\in \mathcal S'$ and $n\in \mathbb Z_{\ge 0}$, then $C$ is a two-sided tilting complex.
\end{proposition}

\section{Perverse equivalences}\label{s:perverse}

We recall the notion of perverse equivalences introduced in \cite{Perverse}.
In this section, we assume that $\Gamma$ and $\Lambda$ are finite dimensional indecomposable symmetric $k$-algebras.
Let $\mathcal{S}$ be a complete set of representatives of isomorphism classes of simple $\Gamma$-modules, $r$ a non-negative integer, $p:\{0,\dots,r\}\to \mathbb{Z}$ a map, and $\mathcal S_{\bullet}$ a filtration of $\mathcal S$ satisfying \[\emptyset = \mathcal{S}_{-1}\subseteq \mathcal{S}_{0} \subseteq \mathcal{S}_{1}\subseteq \dots \subseteq \mathcal{S}_{r}=\mathcal{S}.\]
Similarly, let $\mathcal{S}'$ be a complete set of representatives of isomorphism classes of simple $\Lambda$-modules and $\mathcal S'_{\bullet}$ a filtration of $\mathcal S'$ satisfying \[\emptyset = \mathcal{S}'_{-1}\subseteq \mathcal{S}'_{0} \subseteq \mathcal{S}'_{1}\subseteq \dots \subseteq \mathcal{S}'_{r}=\mathcal{S}'.\]
\begin{definition}\label{def:perverse} An equivalence $F:\derb(\Gamma)\to \derb(\Lambda)$ is \emph{perverse relative to} $(\mathcal{S}_{\bullet}, \mathcal{S}'_{\bullet}, p)$ if for every $i$, the following hold:
  \begin{itemize}
    \item For $S \in \mathcal{S}_i - \mathcal{S}_{i-1}$, the composition factors of $H^{\ell}(F(S))$ are in $\mathcal S'_{i-1}$ for all $\ell\neq -p(i)$.
    \item There is a filtration of $\Lambda$-modules $L_1\subseteq L_2 \subseteq H^{-p(i)}(F(S))$ such that all composition factors of $L_1$ and $H^{-p(i)}(F(S))/L_2$ are in $\mathcal S'_{i-1}$, and $L_2/L_1 \in \mathcal{S}'_i - \mathcal{S}'_{i-1}$.
    \item The map $S\mapsto L_2/L_1$ induces a bijection $\mathcal{S}_i - \mathcal{S}_{i-1} \to \mathcal{S}'_i - \mathcal{S}'_{i-1}$.
  \end{itemize}
\end{definition}
We note that the definition of perverse equivalences presented here is equivalent to the one established by Chuang-Rouquier, as supported by \cite[Lemma 4.19]{Perverse}.
The following proposition indicates that perverse equivalences have good properties.
\begin{proposition}[{\cite[Lemma 4.2, Proposition 4.17. and Lemma 4.16.]{Perverse}}]\label{usefulperverse}
  Let $\Gamma'$ be a finite dimensional indecomposable symmetric $k$-algebra, $F:\derb(\Gamma)\to \derb(\Lambda)$ a perverse equivalence relative to $(\mathcal{S}_{\bullet}, \mathcal{S}'_{\bullet}, p)$ and $G:\derb(\Lambda) \to \derb(\Lambda')$ a perverse equivalence relative to $(\mathcal{S}'_{\bullet}, \mathcal{S}''_{\bullet}, p')$.
  The following hold:
  \begin{enumerate}
    \item\label{usefulperverse1} $F^{-1}$ is perverse relative to $(\mathcal{S}'_{\bullet}, \mathcal{S}_{\bullet}, -p)$.
    \item\label{usefulperverse2} $G\circ F$ is perverse relative to $(\mathcal{S}_{\bullet},\mathcal{S}''_{\bullet},p+p')$.
    \item\label{usefulperverse3} If $p=0$, then
    $F$ restricts to a Morita equivalence from $\Gamma$ to $\Lambda$. \end{enumerate}
\end{proposition}

Next, for $i\in\{0,\dots,r\}$ and $S\in \mathcal{S}_i-\mathcal{S}_{i-1}$, we construct $T_S=(T_S^{\ell}, d^{\ell})_{\ell\in\mathbb Z}$ a complex with zero terms in degrees other than $-r,\dots,-i$, as follows. Put $T_S^{-i}=P(S)$. Having constructed $T_{S}^{u}$ and $d_{S}^u$ for all $u\in\{-j,\dots,-i\}$ for $j\in \{\check{i,\dots, r}\}$, let $M^{-j}$ be the smallest submodule of $K^{-j}:=\Ker(d^{-j}:T_S^{-j}\to T_S^{1-j})$ such that all composition factors of $K^{-j}/M^{-j}$ lie in $\mathcal{S}_j$. Define $d^{-j-1}: T_S^{-j-1}\to T_S^{-j}$ to be the composition of a projective cover $T_{S}^{-j-1}\to M^{-j}$ with the inclusion of $M^{-j}$ into $T_S^{-j}$.

\begin{proposition}[{\cite[Proposition 5.7]{Perverse}\label{decreasing}}]
  The complex $T=\bigoplus_{S\in\mathcal{S}} T_S$ is a tilting complex and the equivalence $F=\Hom_\Gamma^{\bullet}(T,{-}):\derb(\Gamma)\to \derb(\End_{\homob(\projcat \Gamma)}(T))$ is perverse relative to
  $(\mathcal{S}_{\bullet},\mathcal{S}'_{\bullet},p)$, where $p$ is given by $p(i)=-i$.
\end{proposition}

We say that the equivalence $F=\Hom_\Gamma^{\bullet}(T,{-})$ is \emph{decreasing perverse relative} to $\mathcal{S}_{\bullet}$.

Let $\Lambda=\End_{\homob(\projcat \Gamma)}(T)$.
Given $S\in \mathcal{S}$, then $F(T_S)$ is isomorphic to an indecomposable projective $\Lambda$-module whose simple quotient we denote by $S'$.
For $S\in \mathcal{S}_i-\mathcal{S}_{i-1}$, we construct a complex $Y_S=(Y_S^{\ell}, d^{\ell})_{\ell\in \mathbb Z}$ with zero terms in degrees other than $-i,\dots, 0$. If $i=0$, we put $Y_S=S$. Otherwise start by putting $Y_S^{-i}=P(S)$. If $i=1$, we define $d^{-i}: Y_S^{-i}\to Y_S^{1-i}$ to be the quotient map from $P(S)$ to $P(S)/N^{1-i}$, where we define $N^{1-i}$ to be the largest submodule of $P(S)$ such that all composition factors of $N^{1-i}/S$ are in $\mathcal S^{i-1}$. Otherwise, we define $d^{-i}: Y_S^{-i}\to Y_S^{-i+1}$ to be the composition of the quotient map $P(S)\to P(S)/N^{1-i}$ with an injective hull $P(S)/N^{1-i} \to Y_S^{1-i}$.

Having constructed $Y_S^{u}$ and $d_S^{-1+u}$ for all $u\in\{-i,\dots, -j\}$ for $-j\in\{-i,\dots,-1\}$, let $N^{1-j}$ be the largest quotient of $C^{1-j}:=\Coker(d^{-1-j}:Y_S^{-1-j}\to Y_S^{-j})$ such that all composition factor of $\Ker(C^{1-j}\to N^{1-j})$ are in $\mathcal{S}_{j-1}$. Then let $d^{-j}:Y_S^{-j}\to Y_S^{1-j}$ be the composition of the canonical epimorphism $Y_S^{-j}\to N^{1-j}$ with an injective hull $N^{1-j}\to Y_S^{1-j}$. When $j=1$, the constructions of $C^{1-j}$ and $N^{1-j}$ are the same, but $d^{-j}:Y_S^{-j}\to Y_S^{1-j}$ is the canonical epimorphism $Y_S^{-j}\to N^{1-j}=Y_S^{1-j}$. The following proposition holds:
\begin{proposition}[{\cite[Lemma 5.9.]{Perverse}}]\label{simplemindeddecreasing}
  We have $Y_S \cong F^{-1}(S')$ for $S\in \mathcal{S}$.
\end{proposition}

\section{Preliminaries for graphs}\label{s:graphs}
Generalized Brauer tree algebras are constructed by generalized Brauer trees. Here we recall the definition of generalized Brauer trees and introduce terminology of rooted trees on them. The terminology helps us to describe the structures of modules
and homomorphisms over generalized Brauer tree algebras. This also works well for describing some tilting complexes. \begin{definition}[see {\cite[\S 2.2]{Schroll2018}}]
  Let $V$ be a finite set and $E$ a subset of the set \[\{~\!\{v,w\}\mid v,w \in V,\quad v\neq w\}.\]
  We say that $\mathcal T=(V,E)$ is an \emph{undirected graph} of the set of the \emph{vertices} $V$ and the set of the \emph{edges} $E$. Here, multi-edges and loops are not admitted for the graph.

  Let $\mathcal T=(V,E)$ be an undirected graph.
  A \emph{path} from a vertex $v$ to a vertex $w$ is a series of edges $e_1, e_2, \dots, e_N~~(N\in \mathbb N)$ satisfying the following conditions:
  \begin{itemize}
    \item $v\in e_1,~v\notin e_2,$
    \item $\forall i \in \{1,\dots,N-1 \}\quad \lvert e_i\cap e_{i+1}\rvert=1,$
    \item $w\notin e_{N-1},~w\in e_N.$
  \end{itemize}
  The \emph{length} of a path $(e_1,\dots, e_N)$ is defined as the number $N$. For example, the set of all paths of length $1$ is same to the edge set $E$.
  A path from a vertex $v$ to $v$ is called a \emph{cycle} of $v$.

  An undirected graph $\mathcal T=(V,E)$ is called a \emph{connected undirected graph} if there is a path from $v$ to $w$ for any distinct $v,w\in V$.
  We call a connected undirected graph with no cycles a \emph{tree}.

  It is called a \emph{generalized Brauer tree} that a tree $\mathcal T=(V,E)$ with a function $m: V \to \mathbb Z_{>0}$ and a cyclic order $\rho_v$ on the set of edges around $v$ assigned to each vertex $v$. For a vertex $v$, the image $m_v$ of $v$ by $m$ is called the \emph{multiplicity} at $v$.
  We write a generalized Brauer tree $\mathcal T=(V,E)$ with multiplicities $m$ and the cyclic orders $\rho$ by $(V,E,m,\rho)$.
\end{definition}

\begin{example}\label{ex:generalizedBrauertree}
  Assume \[V = \{a,b,c,d,e,f,g\},\]
  \[E=\bigl\{1=\{a,c\},~2=\{b,c\},~3=\{c,g\},~4=\{d,g\},~5=\{e,g\},~6=\{f,g\}\bigr\},\]
  \[m_a=m_d=1,\quad m_c=m_e=m_g=2,\quad m_b=m_f=3,\]
  and
  \[\begin{split}
      \rho_a =(1),\quad \rho_b&=(2), \quad \rho_c =(1<2<3),\quad \rho_d=(4),\\ &\rho_e=(5),\quad \rho_f = (6), \quad \rho_g=(3<4<5<6).
    \end{split}\]
  Then $\mathcal T=(V,E,m,\rho)$ is a generalized Brauer tree. We can depict this as follows:
  \tikzset{maru/.style={{circle,draw,minimum width=0, line width=0,minimum height=0,inner sep=1.5}}}
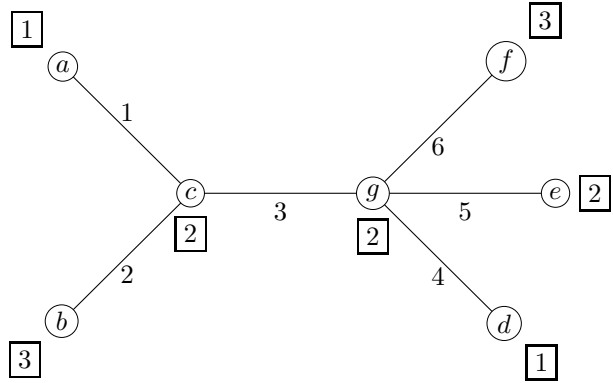
  \[\begin{tikzpicture}
      \node[maru,label=below:\fbox{2}](C){$c$};
      \node[above left = 2 of C, maru,label=above left:\fbox{1}](A){$a$};
      \node[below left = 2 of C, maru, label = below left:\fbox{3}](B){$b$}; \node[right = 2 of C, maru,label=below:\fbox{2}](D){$g$};
      \node[below right = 2 of D,maru,label= below right:\fbox{1}](E){$d$};
      \node[right = 2 of D,maru,label=right:\fbox{2}](F){$e$};
      \node[above right = 2 of D,maru, label= above right:\fbox{3}](G){$f$};
      \path[-,draw](C) to node[above]{$1$} (A);
      \path[-,draw](C) to node[below]{$2$} (B);
      \path[-,draw](C) to node[below]{$3$} (D);
      \path[-,draw](D) to node[below]{$4$} (E);
      \path[-,draw](D) to node[below]{$5$} (F);
      \path[-,draw](D) to node[below]{$6$} (G);
    \end{tikzpicture}\]
\end{example}
\begin{remark}
  The cyclic order at the vertex $g$ in \cref{ex:generalizedBrauertree} is the same as
  \[(4<5<6<3),\qquad (5<6<3<4),\qquad (6<3<4<5).\]
\end{remark}
\begin{remark}
  Let $\mathcal T=(V,E,m,\rho)$ be a generalized Brauer tree. Then $\mathcal T$ is a planer graph. Namely, it can be embedded in the plane. For a vertex $v$, we read the cyclic order on $v$ as the counterclockwise of the edges emanating from $v$.
\end{remark}

We introduce terminology of rooted trees to generalized Brauer tree algebras.
\begin{definition}\label{terminology}
  Let $\mathcal T=(V,E,m,\rho)$ be a generalized Brauer tree.
  Fix $v_0 \in V$. We call the pair $(\mathcal T, v_0)$ a \emph{rooted generalized Brauer tree} and we call $v_0$ a \emph{root} of $(\mathcal T,v_0)$. For a vertex $v$ of $(\mathcal T, v_0)$ other than $v_0$, the length of the path from $v_0$ to $v$ is called the \emph{depth} of $v$ we denote by $d(v)$. We define the depth of $v_0$ as $0$. The maximum of $d(w)$ for all $w\in V$ is called the \emph{height} of $(\mathcal T,v_0)$.
  The number of edges incident to $v$ is called the \emph{degree} of $v$ and is denoted by $\deg(v)$. The degree of $v_0$ is defined in the same manner.
  Given $(e_1,\dots,e_{d(v)})$ the path from $v$ to $v_0$,
  the vertex $w$ satisfying $\{v,w\}= e_1$ is called a \emph{parent} of $v$
  and is denoted by $p(v)$. The set of $w'\in V$ comprising $p(w')=v$ is called the set of \emph{children} of $v$.
  We take the cyclic order $\rho_v$ around $v$
  \[(\rho_v^1<\rho_v^2<\dots<\rho_v^{\deg(v)})\]
  to hold $\rho_v^{\deg(v)}=\{v,p(v)\}$.
  The edge $\rho_v^j$ is called the \emph{$j$th child} of $v$ we denote by $c(v,j)$ for $j\in \{1,\dots,\deg(v)-1\}$. We define a map $s$ from the vertices other than $v_0$ to $\mathbb N$ to hold $c(p(v),s(v))=v$ for all vertices $v$ other than $v_0$. We call the map $s$ as a \emph{sibling index map}.

  We have a one-to-one correspondence by
  \[\begin{tikzcd}[ampersand replacement=\&, row sep = tiny]
      V\rar[] \& E~\sqcup~\{0\} \\
      \rotatebox{90}{$\in$} \& \rotatebox{90}{$\in$} \\
      v\rar[mapsto]\&
      \begin{cases}
        0      & \text{if $v=v_0$,}         \\
        n\in E & \text{if $n= \{v,p(v)\}$.}
      \end{cases}
    \end{tikzcd}\]
  By this correspondence, we shall extend the depth function $d$, the degree function $\deg{}$, the parent map $p$, the $j$th child $c( {-}, j)$ and a sibling index map $s$,
  to make the following diagrams commutative:
  \[
    \begin{tikzcd}[sep=large]
      V\rar["d"]\dar[leftrightarrow, ""{name=U}]& \mathbb{N}\dar[equal, ""{name=V}]\\
      E~\sqcup~\{0\}\rar["d",dashed]&\mathbb N
      \ar[from=U, to=V, phantom, "\circlearrowright"]
    \end{tikzcd}, \quad
    \begin{tikzcd}[sep=large]
      V\rar["\deg{}"]\dar[leftrightarrow,""{name=U}]& \mathbb{N}\dar[equal,""{name=V}]\\
      E~\sqcup~\{0\}\rar["\deg{}",dashed]&\mathbb N
      \ar[from=U, to=V, phantom, "\circlearrowright"]
      \ar[from=U, to=V, phantom, "\circlearrowright"]
    \end{tikzcd}, \quad
    \begin{tikzcd}[sep=large]
      V-\{v_0\}\rar["p"]\dar[leftrightarrow,""{name=U}]& V\dar[leftrightarrow,""{name=V}]\\
      E\rar["p",dashed]&E \sqcup \{0\}
      \ar[from=U, to=V, phantom, "\circlearrowright"]
    \end{tikzcd},\]\[
    \begin{tikzcd}[sep=large]
      V\rar["{c({-},~j)}",dashed]\dar[leftrightarrow,""{name=U}]& V\dar[leftrightarrow,""{name=V}]\\
      E~\sqcup~\{0\}\rar["{c({-},~j)}"]&E~\sqcup~\{0\}
      \ar[from=U, to=V, phantom, "\circlearrowright"]
    \end{tikzcd},\quad
    \begin{tikzcd}[sep=large]
      V\rar["s"]\dar[leftrightarrow,""name=U]& \mathbb N \dar[equal,""name=V]\\
      E~\sqcup~\{0\}\rar["s",dashed]&\mathbb N
      \ar[from=U, to=V, phantom, "\circlearrowright"]
    \end{tikzcd}.\]
\end{definition}
We define some orders determined by Green's walk on generalized Brauer trees.
\begin{definition}
  Let $(\mathcal T,v_0) = (V,E, m,\rho, v_0)$ be a rooted generalized Brauer tree. We define a transformation $T_v$ on ordered lists of vertices for $v\in V$ as the following inductive way:
  \begin{itemize}
    \item If $v=v_0$, for an ordered list of vertices $L$, we set $L'$ to be the list obtained by appending $v$ at the end of the list $L$. If $v$ has a child, we define \[T_v(L)=(T_{c(v,\deg(v))}\circ\dots\circ T_{c(v,1)})(L').\] Otherwise, $T_v(L)=L'$.
    \item If $v\neq v_0$, for an ordered list of vertices $L$, we set $L'$ to be the list obtained by appending $v$ at the end of the list $L$. If $v$ has a child, we define \[T_v(L)=(T_{c(v,\deg(v)-1)}\circ\dots\circ T_{c(v,1)})(L').\] Otherwise, $T_v(L)=L'$.
  \end{itemize}
  Here, we take $L$ as an empty list. Then, we get the list $T_{v_0}(L)$. Each vertex appears in the list just once. This list is called the \emph{preorder traversal of vertices} of $(\mathcal T,v_0)$. By taking the images of vertices of the list through multiplicities $m$, we have a list $L_m$ consisting of natural numbers. We say that the list $L_m$ is \emph{preorder traversal of multiplicities}. By deleting the first element $v_0$ of $T_{v_0}(())$ and then applying the one-to-one correspondence from $V$ to $E\sqcup \{0\}$, described in \cref{terminology}, we have a list $L_E$ consisting of edges. We say that the list $L_E$ is the \emph{preorder traversal of edges} of $(\mathcal T,v_0)$.
\end{definition}

\begin{example}\label{myexample}
  Assume
  \[V = \{v_0,v_1, v_2, v_3, v_4, v_5, v_{6}\},\]
  \[\begin{split}
      E=\bigl\{1=\{v_0,v_1\},~2&=\{v_1,v_2\},~3=\{v_1,v_3\},\\
      ~4=\{v_3,v_4\},~5&=\{v_1,v_5\},~6=\{v_0,v_6\}\bigr\},
    \end{split}\]
  \begin{gather*}
    m_{v_0}=m_{v_{1}}=m_{v_{2}}=m_{v_{4}}=1,\\ m_{v_3}=m_{v_6}=2,~~ m_{v_5}=3,
  \end{gather*}
  and the cyclic order $\rho_v$ around $v$ is \[(\rho_v^1<\rho_v^2<\dots<\rho_v^{\deg(v)})\] for all $v\in V$ satisfying
  \[\begin{split}
      \rho_{v_1}^1&=2,~\rho_{v_1}^2=3,~\rho_{v_1}^3=5,~\rho_{v_1}^4=1, \\
      \rho_{v_2}^1&=2,~ \rho_{v_3}^1=4,~\rho_{v_3}^2=3,~
      \rho_{v_4}^1=4,\\~\rho_{v_5}^1&=5,~\rho_{v_6}^1=6,~
      \rho_{v_0}^1=1,~\rho_{v_0}^2=6.
    \end{split}\]
  Then $\mathcal T=(V,E,m,\rho)$ is a generalized Brauer tree. We can depict a rooted generalized Brauer tree $(\mathcal T,v_0)$ as follows:
  
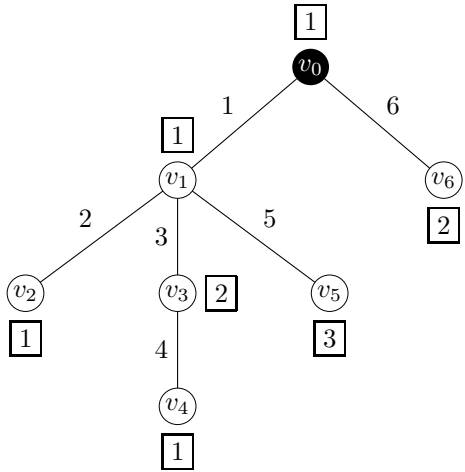
\[
    \begin{tikzpicture}
      \tikzset{block/.style={circle,draw,thick, minimum width=0, line width=0,minimum height=0,inner sep=1.0}};
      \tikzset{block1/.style={circle, white, fill=black,minimum width=0, line width=0,minimum height=0,inner sep=1.0}}
      \node[block1,label=above:\fbox{1}]{$v_0$}[sibling distance=3.5cm]
      child{ node[block, label=above:\fbox{1}] {$v_1$}[sibling distance=2cm]
          child{ node[block, label=below:\fbox{1}] {$v_2$}edge from parent node [above left] {2}}
          child{ node[block, label=east:\fbox{2}] (c) {$v_3$}[sibling distance=1.5cm]
              child{ node[block, label=below :\fbox{1}] {$v_4$} edge from parent node [left] {4}} edge from parent node [left]{3}}
          child{ node[block, label=below:\fbox{3}] {$v_5$}
              edge from parent node [above right] {5}
            } edge from parent node [above left] {1}
        }
      child{ node[block, label=south:\fbox{2}] {$v_6$} edge from parent node [above right] {6}}
      ;
    \end{tikzpicture}\]
  We remark that the marked vertex $v_0$ indicates the root of the generalized Brauer tree.
  We have
  \[d(v_3)=d(3)=2,~\deg(v_1)=4,~p(v_4)=v_3=c(v_1,2),~p_4=3=c(1,2),~ s(v_5) =s(5)=3.\]
  The preorder traversal of vertices of $(\mathcal T, v_0)$ is
  \[(v_0,v_1, v_2, v_3, v_4, v_5, v_6).\]
  The preorder traversal of edges of $(\mathcal T, v_0)$ is
  \[(1,2,3,4,5,6).\]
  The preorder traversal of multiplicities of $(\mathcal T, v_0)$ is
  \[(1,1,1,2,1,3,2).\]
\end{example}

\section{Generalized Brauer tree algebras}\label{s:gbta}

We construct generalized Brauer tree algebras by using generalized Brauer trees. We introduce some notion to define generalized Brauer tree algebras.

Let $\mathcal T=(V,E,m,\rho)$ be a generalized Brauer tree.
We take the cyclic order $\rho_v$ around $v$
\[(\rho_v^1<\rho_v^2<\dots<\rho_v^{\deg(v)}),\]
here we note that $\deg(v)$ is equal to the number of edges emanating from $v$ for all $v \in V$.
We put
\begin{gather*}
  Q_0:=E,\\
  Q_1:=\{~\gamma_{v}^t~\mid~ t\in \{1,\dots,\deg(v)\},\quad v\in V\},
\end{gather*}
where we define $\gamma_{v}^t$ to be an arrow from
$\rho_{v}^{t-1}$ to $\rho_{v}^{t}$ for $t\in
  \{2,\dots,\deg(v)\}$ and $\gamma_{v}^{1}$ to be an arrow from $\rho_{v}^{\deg(v)}$ to $\rho_{v}^{1}$.

We put $Q=(Q_0,Q_1)$.
We set $A_0=kQ$ as the path algebra of the quiver $Q$. Let $I$ denote the ideal of $A_0$ generated by the following elements:
\begin{itemize}
  \item $\gamma_v^t\gamma_w^u$ ~~for $v\neq w \in V$, $t\in\ \{1,\dots,\deg(v)\}$ and $u\in\ \{1,\dots,\deg(w)\}.$
  \item $(\gamma_v^{t}\gamma_v^{t+1}\dots\gamma_v^{\deg(i)}\gamma_v^{1}\dots \gamma_{v}^{t-1})^{m_v}\gamma_v^{t}$ ~~for $v\in V$ and $t\in\ \{1,\dots,\deg(v)\}$.
  \item $(\gamma_v^{t}\gamma_v^{t+1}\dots\gamma_v^{\deg(v)}\gamma_v^{1}\dots \gamma_{v}^{t-1})^{m_v}-(\gamma_w^{u}\gamma_w^{u+1}\dots\gamma_w^{\deg(w)}\gamma_w^{1}\dots \gamma_{w}^{u-1})^{m_w}$\\ for $\{v,w\}\in E$, $t\in\ \{1,\dots,\deg(v)\}$ and $u\in\ \{1,\dots,\deg(w)\}$ such that the both source of $\gamma_v^{t}$ and $\gamma_w^{u}$ are equal to the edge $\{v,w\}$.
\end{itemize}
\begin{definition}[{see \cite[\S2.3 and \S 2.4]{Schroll2018}}]\label{def:quiverandalgebrafromgeneralizedBrauertreealgebra}
  The quiver $Q=(Q_0,Q_1)$ is called a \emph{Brauer quiver} of the generalized Brauer tree $\mathcal T$.
  We say that the admissible ideal $I$ is the relation ideal associated to $\mathcal{T}$. The bound quiver algebra $A=A_0/I$ is called the \emph{generalized Brauer tree algebra} associated to the generalized Brauer tree $\mathcal T$.
\end{definition}

\begin{example}
  Let $\mathcal T$ be the generalized Brauer tree described in \cref{ex:generalizedBrauertree}. The generalized Brauer quiver of $\mathcal T$ is as follows:
  \tikzset{maru/.style={{circle,draw,minimum width=13pt, line width=0,minimum height=0,inner sep=1.5}}}

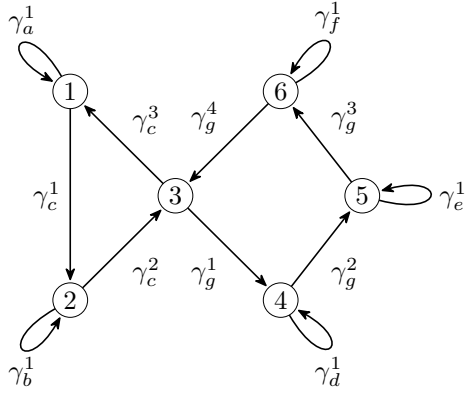
  \[\begin{tikzpicture}[remember picture,->,>={Stealth[round]},shorten >=1pt, node distance=2cm,semithick]
      \node[label=below:, maru](C){3};
      \node[above left = 1.5 of C,label=above left:, maru](A){1};
      \node[below left = 1.5 of C, label = below left:, maru](B){2}; \node[below right = 1.5 of C, label= below right:,maru](E){4};
      \node[right = 2 of C,label=right:, maru](F){5};
      \node[above right = 1.5 of C, label= above right:, maru](G){6};
      \path[->,draw](A) to node[left]{$\gamma_c^1$} (B);
      \path[->,draw](B) to node[below right]{$\gamma_c^2$} (C);
      \path[->,draw](C) to node[above right]{$\gamma_c^3$} (A);
      \path[->,draw](C) to node[below left]{$\gamma_g^1$} (E);
      \path[->,draw](E) to node[below right]{$\gamma_g^2$} (F);
      \path[->,draw](F) to node[above right]{$\gamma_g^3$} (G);
      \path[->,draw](G) to node[above left]{$\gamma_g^4$} (C);

      \path[->,draw] (A) edge [out=120,in=150,looseness=20] node[above] {$\gamma_a^1$} (A)
      (B) edge [out=-150,in=-120,looseness=20] node[below] {$\gamma_b^1$} (B)
      (E) edge [out=-60,in=-30,looseness=20] node[below] {$\gamma_d^1$} (E)
      (F) edge[out=-15,in=15,looseness=20] node[right] {$\gamma_e^1$} (F)
      (G) edge[out=30,in=60,looseness=20] node[above]{$\gamma_f^1$} (G); \end{tikzpicture}
  \]

  The admissible ideal obtained from $\mathcal T$ is generated by
  \begin{gather*}
    \gamma_a^1\gamma_c^1,~\gamma_c^1\gamma_b^1,~\gamma_b^1\gamma_c^2,~\gamma_c^3\gamma_a^1,~ \gamma_c^2\gamma_g^1,~ \gamma_g^1\gamma_d^1,~\gamma_d^1\gamma_g^2,~\gamma_g^2\gamma_e^1,~\gamma_e^1\gamma_g^3,~\gamma_g^3\gamma_f^1,~\gamma_f^1\gamma_g^4,~\gamma_g^4\gamma_c^3,\\
    (\gamma_a^1)^2,~(\gamma_c^1\gamma_c^2\gamma_c^3)^2 \gamma_c^1,~(\gamma_b^1)^4,~ (\gamma_c^2\gamma_c^3\gamma_c^1)^2 \gamma_c^2,~
    (\gamma_c^3\gamma_c^1\gamma_c^2)^2 \gamma_c^3,~
    (\gamma_g^1\gamma_g^2\gamma_g^3\gamma_g^4)^2 \gamma_g^1,\\
    (\gamma_g^2\gamma_g^3\gamma_g^4\gamma_g^1)^2 \gamma_g^2,~(\gamma_d^1)^2,~(\gamma_g^3\gamma_g^4\gamma_g^1\gamma_g^2)^2\gamma_g^3,~(\gamma_e^1)^3,~(\gamma_g^4\gamma_g^1\gamma_g^2\gamma_g^3)^2 \gamma_g^4,~(\gamma_f^1)^4,\\
    \gamma_a^1-(\gamma_c^1\gamma_c^2\gamma_c^3)^2,~(\gamma_b^1)^3-(\gamma_c^2\gamma_c^3\gamma_c^1)^2,~(\gamma_c^3\gamma_c^1\gamma_c^2)^2-(\gamma_g^1\gamma_g^2\gamma_g^3\gamma_g^4)^2,\\~(\gamma_g^2\gamma_g^3\gamma_g^4\gamma_g^1)^2-\gamma_d^1,~(\gamma_g^3\gamma_g^4\gamma_g^1\gamma_g^2)^2-(\gamma_e^1)^2,~(\gamma_g^4\gamma_g^1\gamma_g^2\gamma_g^3)^2-(\gamma_f^1)^3.
  \end{gather*}
\end{example}

\begin{proposition}[see \textup{\cite[Theorem 2.6]{Schroll2018}}] Generalized Brauer tree algebras are finite dimensional indecomposable symmetric special biserial algebras of tame representation type.
\end{proposition}

\section{Modules over a rooted generalized Brauer tree algebra}\label{s:modules}
We describe projective and simple modules over a generalized Brauer tree algebra associated to a rooted generalized Brauer tree using the terminology introduced in \cref{terminology}.
Let $k$ be an algebraically closed field and $A$ a generalized Brauer tree $k$-algebra associated to a generalized Brauer tree $\mathcal T=(V,E,m,\rho)$. Fix a vertex $v_0$ of $\mathcal T$.
Let $Q=(Q_0,Q_1)$ be a generalized Brauer quiver of $A$. Let $I$ denote the relation ideal associated to $\mathcal T$.

We take the depth function $d$, the degree function $\deg$, the parent map $p$, the $j$th child $c( {-}, j)$ and a sibling index map $s$ of $(\mathcal T, v_0)$. We set the vertex $v_n$ to hold $\{v_n,p(v_n)\}=n$ for all $n\in E$.
For all $v\in V$,
we take the cyclic order $\rho_v$ around $v$
\[(\rho_v^1<\rho_v^2<\dots<\rho_v^{\deg(v)})\]
to hold $\rho_v^{\deg(v)}=\{v,p(v)\}$. We have
\begin{gather*}
  Q_0=E,\\
  Q_1=\{~\gamma_{v}^t~\mid~ t\in \{1,\dots,\deg(v)\},\quad v\in V\},
\end{gather*}
where we define $\gamma_{v}^t$ is an arrow from
$\rho_{v}^{t-1}$ to $\rho_{v}^{t}$ for $t\in
  \{2,\dots,\deg(v)\}$ and $\gamma_{v}^{1}$ is an arrow from $\rho_{v}^{\deg(v)}$ to $\rho_{v}^{1}$.

For $v\in V$ and $j\in \{1,\dots,\deg(v)\}$, we define $c(v,0)=v$, $c(v,j+\deg(v)\ell)= c(v,j)$ and $\gamma_v^{j+\deg(v)\ell}=\gamma_v^j$ for all $\ell\in \mathbb Z$ for ease of notation. Using an abuse of notation, for $n\in E$ and $j\in \mathbb Z$, we represent $\gamma_{v_n}^j$ as $\gamma_n^j$ and $\gamma_{v_0}^j$ as $\gamma_0^j$.
Similarly, for $n\in E$ and $j\in \{1,\dots,\deg(n)\}$, we define $c(n,0)=n$, $c(n,j+\deg(n)\ell)= c(n,j)$ and $\gamma_n^{j+\deg(n)\ell}=\gamma_n^j$ for all $\ell\in \mathbb Z$ for convenience.

We denote by $\varepsilon_n \in kQ$ the stationary path of $n$ for all $n \in Q_0$. The set $\{e_n := \varepsilon_n + I \mid n \in Q_0\}$ is a complete set of primitive
orthogonal idempotents of the bound quiver algebra $A$. We denote by $P_n$ an indecomposable projective $A$-module $e_nA$ and by $S_n$
a quotient module of $P_n$ by the radical $\rad(P_n)$, which is a simple module indexed by $n$ for every $n\in E=Q_0$. We remark that
the set $\{P_n \mid n \in Q_0\}$ is a complete set of all indecomposable projective $A$-modules up to isomorphism. The set $\{S_n \mid n \in Q_0\}$ is a complete set of all simple $A$-modules up to isomorphism.
We denote by $[\gamma_n^j:\ell]$ the path $\gamma_n^j\gamma_n^{j+1}\gamma_n^{j+2}\dots\gamma_n^{j+\ell-1}$ for every $n\in Q_0$ and $j, \ell> 0$.
We denote the vector space generated by a subset $X$ of $A$ by $\langle X \rangle_k$.
We represent $A$-modules as quotients of certain direct sums of vector spaces contained in $A$, obtained by right-multiplying elements of $A$ if they are closed under such multiplication. To describe the radical series of an $A$-module, we use a boxed column where each entry represents the index of a simple module.
\begin{remark}\label{indecprojstr}
  For each $n\in E$, it holds that $P_n/\rad(P_n)
    \cong S_n$, \[\soc(P_n)\cong \langle \{\,[\gamma_{p_n}^{s(n)+1}:\deg(p_n)m_{p_n}]+I=[\gamma_{n}^1:\deg(n)m_{n}]+I\}\rangle_k \cong S_n\] and $\rad(P_n)/\soc(P_n)$ is isomorphic to a direct sum of two uniserial modules we denote by $U_n^{+}$ and $U_n^{-}$:
  \begin{gather*}
    U_n^{+}=\langle \{\,[\gamma_{p_n}^{s(n)+1}:\ell]+I+\soc(P_n)\,\mid\, \ell\in\{1,\dots,\deg(p_n)m_{p_n}-1\}\}\,\rangle_k \\
    \cong
    \fbox{$\begin{matrix}
          {c(p_n,s(n)+1)}\\{c(p_n,s(n)+2)} \\\vdots\\ {c(p_n,\deg(p_n)m_{p_n}+s(n)-2)}\\{c(p_n,\deg(p_n)m_{p_n}+s(n)-1)}
        \end{matrix}$},\\
    U_n^{-}=\langle \{\,[\gamma_{n}^1:\ell] + I +\soc(P_n)\,\mid\, \ell\in\{1,\dots,\deg(n)m_{n}-1\}\}\,\rangle_k \\
    \cong \,
    \fbox{$\begin{matrix}
          {c(n,1)}\\{c(n,2)} \\\vdots\\ {c(n,\deg(n)m_{n}-2)}\\{c(n,\deg(n)m_{n}-1)}
        \end{matrix}$}\,.
  \end{gather*}
\end{remark}

Given $i,j\in Q_0$, let $[v_j\to v_i]$ denote the set of all paths from $v_j$ to $v_i$. We have the following $k$-isomorphism:
\[
  \begin{tikzcd}[row sep=tiny]
    \pathtohom{}: \langle~{[v_j\to v_i]}~\rangle_k/\varepsilon_j I \varepsilon_i \rar[]& \Hom_{A}(P_i,P_j)\\
    \rotatebox{90}{$\in$}& \rotatebox{90}{$\in$}\\
    \sigma + \varepsilon_j I \varepsilon_i\rar[mapsto]& (e_i a~{\mapsto}~ \sigma e_i a)
  \end{tikzcd}\]
We call $\pathtohom{(\sigma+\varepsilon_j I \varepsilon_i)}$ the homomorphism corresponding to $\sigma+\varepsilon_j I \varepsilon_i$.
By abuse of notation,
we denote $\pathtohom{(\sigma+\varepsilon_j I \varepsilon_i)}$ by $\pathtohom{\sigma}$.

\begin{proposition}\label{dimHom}
  We have
  \[\dim \Hom_A(P_i, P_j)=\begin{cases}
      0           & (i\cap j= \emptyset), \\ m_v & (i\cap j= \{v\}),\\
      m_v + m_{w} & (i=j =\{v,w\}),
    \end{cases}\]
  for $i,j\in E$.
\end{proposition}
\begin{proof}
  The fact follows from counting the composition factors isomorphic to $S_i$ of the projective module $P_j$. The structures of indecomposable projective modules are described in \cref{indecprojstr}.
\end{proof}

\section{Main results}\label{s:main}

We use the hypotheses and notation made in \cref{s:modules}.
Let $d$ be a depth function and $r$ the height of $(\mathcal T,v_0)$. We notice that the domain of the function $d$ is extended to the set of edges.
Let $\mathcal{S}$ be a complete set of representatives of isomorphism classes of simple $A$-modules.
We take \[\mathcal{S}_i=\{S_n \in \mathcal{S} \mid d(n)\ge r-i\}\]
for $i\in\{-1,0,1, \dots, r-1\}$.
Then, we have a filtration \[\mathcal{S}_{\bullet}:\emptyset = \mathcal{S}_{-1}\subseteq \mathcal{S}_{0} \subseteq \mathcal{S}_{1}\subseteq \dots \subseteq \mathcal{S}_{r-1}=\mathcal{S}\]
of $\mathcal{S}$. We remark that $S_n\in \mathcal S_i-\mathcal S_{i-1}$ is equivalent to $r-d(n)=i$ for all $n\in E$ and $i\in\{0,\dots,r-1\}$.
In this case, the composition factors of $U_n^+$ are in $\mathcal S_i$ but not in $\mathcal S_{i-1}$ except for the composition factor $S_{p_n}\in\mathcal S_{i+1}$. The composition factors of $U_n^{-}$ are in $\mathcal S_{i-1}$ but not in $\mathcal S_{i-2}$ except for the composition factor $S_n\in \mathcal S_{i}$.

Let $p$ be a parent map and $s$ a sibling index map of $(\mathcal{T},v_0)$. For convenience, we define $p^0$ as the identity map and define inductively $p^{\ell+1}$ as the composition of maps $p$ with $p^{\ell}$ for $\ell \ge 0$.
Given $i\in\{0,\dots,r-1\}$ and $n\in E$ satisfying $S_n\in \mathcal S_i-\mathcal S_{i-1}$, we define a complex $T_{n}=(T_n^{\ell}, d_n^{\ell})_{\ell \in \mathbb Z}$ of projective $A$-modules described as follows:
\begin{itemize}
  \item If $\ell < -(r-1)$ or $-i<\ell$, then $T_{n}^{\ell}=0$ and $d_{n}^{\ell}=0$.
  \item If $\ell=-i$, then $T_{n}^{\ell}=P_{n}$ and $d_{n}^{\ell}=0$.
  \item If $-(r-1)\le\ell< -i$, then $T_{n}^{\ell}=P_{p^{-\ell -i}_n}$ and $d_{n}^{\ell}:T_n^{\ell}\to T_n^{\ell+1}$ is a minimal right $\addcat(\{P_n\mid S_n\in \mathcal S_{-\ell}\})$-approximation,
        explicitly a homomorphism
        \[\pathtohom{[\gamma_{p^{-\ell -i}_n}^{s(p^{-\ell-i-1}_n)+1}:\deg(p^{-\ell-i}_n)-(s(p^{-\ell-i-1}_n)+1)]}.\]

\end{itemize}
\begin{proposition}\label{treetostarcomplexisdecreasingperverse}
  The complex $T=\bigoplus_{n=1}^N T_n$ is a tilting complex. An equivalence $F:=\Hom_A^{\bullet}(T,{-}):\derb(A)\to \derb(\End_{\homob(\projcat A)}(T))$ induces the decreasing perverse with respect to the filtration $\mathcal{S}_{\bullet}$.
\end{proposition}
\begin{proof}
  For any $S_n\in\mathcal S_i - \mathcal S_{i-1}$, it is enough to show that $T_n$ is constructed in the same manner as the complex in \cref{decreasing}. First of all, $T_n=(T_n^{\ell}, d_n^{\ell})_{\ell\in\mathbb Z}$ is a complex with nonzero terms in degrees $-(r-1),\dots,-i$. Put $T_n^{-i}=P(S_n)=P_n$. For $\ell \in \{-(r-1),\dots, -i\}$, having constructed $T_{n}^{u}$ for all $u\in \{\ell, \dots, -i\}$, let $M^{\ell}$ be \[\langle \{\,[\gamma_{p^{-\ell-i+1}_n}^{s(p_n^{-\ell-i})+1}:j] + I\,\mid\, j\in\{\deg(p^{-\ell-i+1}_n)-s(p^{-\ell-i}_n),\dots,\deg(p^{-\ell-i+1}_n)m_{p^{-\ell-i+1}_n}\}\}\,\rangle_k \]\[\cong\,
    \fbox{$\begin{matrix}
          {c(p^{-\ell-i+1}_n,\deg(p^{-\ell-i+1}_n))} \\ {c(p^{-\ell-i+1}_n,\deg(p^{-\ell-i}_n)+1)}\\\vdots\\ {c(p^{-\ell-i+1}_n,s(p_n^{-\ell-i})+\deg(p^{-\ell-i+1}_n)m_{p^{-\ell-i+1}_n}-1)}\\{c(p^{-\ell-i+1}_n,s(p_n^{-\ell-i})+\deg(p^{-\ell-i+1}_n)m_{p^{-\ell-i+1}_n})}
        \end{matrix}$}\,.\]
  Define $d_n^{\ell-1}: T_n^{\ell-1}\to T_n^{\ell}$ to be the composition of a projective cover $T_{n}^{\ell-1}\to M^{\ell}$ with the inclusion of $M^{\ell}$ into $T_n^{\ell}$. This shows that $T_n^{\ell-1}$ is isomorphic to $P_{c(p_n^{-\ell-i+1}, \deg(p_n^{-\ell-i+1}))}=P_{p_n^{-\ell-i+1}}$ and $d_n^{\ell-1}$ is a minimal right $\addcat(\{P_n\mid S_n\in \mathcal S_{-\ell}\})$-approximation by the definition of $K^{\ell}$ and the proof of \cref{dimHom}. Then $M^{\ell}$ is the smallest submodule of $K^{\ell}:=\Ker(d_n^{\ell}:T_n^{\ell}\to T_n^{\ell+1})$ such that all composition factors of $K^{\ell}/M^{\ell}$ lie in $\mathcal{S}_{-\ell}$ since if $\ell \neq -i$, $K^{\ell}$ is an indecomposable $A$-module which is the direct sum of $k$-vector spaces
  \begin{align*}
    \langle \{ & \,[\gamma_{p^{-\ell-i+1}_n}^{s(p_n^{-\ell-i})+1}:j] + I \,\mid\, \\ &j\in\{(m_{p^{-\ell-i+1}_n}-1)\deg(p^{-\ell-i+1}_n)+1,\dots,m_{p^{-\ell-i+1}_n}\deg(p^{-\ell-i+1}_n)-s(p^{-\ell-i}_n)-1\}\,\}\rangle_k,
  \end{align*}
  \[\langle \{\,[\gamma_{p^{-\ell-i}_n}^{1}:j] + I \,\mid\, j\in\{1,\dots,\deg(p^{-\ell-i}_n)m_{p^{-\ell-i}_n}-1\}\}\,\rangle_k \]
  and \[\soc(K^{\ell})=\langle \{\,[\gamma_{p_n^{-\ell-i}}^1:\deg(p_n^{-\ell-i})m_{p_n^{-\ell-i}}]+I\}\rangle_k.\]
  We note that $K^{\ell}/\soc(K^{\ell})$ is the direct sum of two uniserial $A$-modules as follows:
  \[\fbox{$\begin{matrix}
          {c(p^{-\ell-i+1}_n,s(p^{-\ell-i}_n)+1)} \\ {c(p^{-\ell-i+1}_n,s(p^{-\ell-i}_n)+2)}\\\vdots\\ {c(p^{-\ell-i+1}_n,\deg(p^{-\ell-i+1}_n)-2)}\\{c(p^{-\ell-i+1}_n,\deg(p^{-\ell-i}_n)-1)}
        \end{matrix}$}\, ,\quad
    \fbox{$\begin{matrix}
          {c(p^{-\ell-i}_n,1)} \\ {c(p^{-\ell-i}_n,2)}\\\vdots\\ {c(p^{-\ell-i}_n,\deg(p^{-\ell-i}_n)m_{p^{-\ell-i}_n}-2)}\\{c(p^{-\ell-i}_n,\deg(p^{-\ell-i}_n)m_{p^{-\ell-i}_n}-1)}
        \end{matrix}$}\,.\]

\end{proof}
We put $B=F(T)$. Two algebras $A$ and $B$ are derived equivalent by \cref{prop:Rickard}. For $S_n\in \mathcal{S}$, we denote $F(T_n)$ by $Q_n$. The module $Q_n$ is an indecomposable projective $B$-module. We denote $Q_n/\rad(Q_n)$ by $V_n$. The module $V_n$ is a simple $B$-module. The construction of $T$ is the same as that stated in \cite[Sec.~4]{MEMBRILLOHERNANDEZ1997231}.
\begin{proposition}[{\cite[Theorem 7.3.]{MEMBRILLOHERNANDEZ1997231}}]\label{perverseandtreetostar}
  The complex $T$ is a tree-to-star tilting complex, as identified in the work of \cite{MEMBRILLOHERNANDEZ1997231}.
\end{proposition}
\begin{remark}[{see \cite{MEMBRILLOHERNANDEZ1997231}}]\label{rem:endomorphism}
  Here, we describe the structure of $B$.
  Let $(e_1,\dots,e_N)$ denote the preorder traversal of edges of $(\mathcal T,v_0)$, where $N$ denotes the number of elements of $E$. We set $e_{N+1}=e_1$ as a matter of convenience.
  For $t\in\{1,\dots,N\}$, we define $\tau_{e_t}=(\tau_{e_t}^{(\ell)})_{\ell\in \mathbb Z}\in \Hom_{\homob(\projcat A)}(T_{e_{t+1}}, T_{e_t})$. If $d(e_{t+1})\le d(e_t)$, we set
  \[\tau_{e_t}^{\ell}=\begin{cases}
      \mathrm{id}                                  & \text{(for $1-r\le \ell < d(e_t)-r$)}, \\
      \pathtohom{\gamma_{p(e_{t+1})}^{s(e_{t+1})}} & \text{(for $\ell=d(e_t)-r$)},          \\
      0                                            & \text{(otherwise)}.
    \end{cases}\]
  If $d(n')> d(n)$, meaning $d(n')=d(n)+1$, we set
  \[\tau_{e_t}^{(\ell)}=\begin{cases}
      \mathrm{id} & \text{(for $1-r\le \ell \le d(e_t)-r$)}, \\
      0           & \text{(otherwise)}.
    \end{cases}\]
  The map $F(\tau_{e_t}): F(T_{e_{t+1}})\to F(T_{e_t})$, ($b\mapsto \tau_n\circ b$) is an irreducible map over $B$-module from $Q_{e_{t+1}}$ to $Q_{e_t}$.
  Similarly, we define $\theta_{e_t}=(\theta_{e_t}^{(\ell)})_{\ell\in\mathbb Z}\in \Hom(T_{e_t}, T_{e_t})$ to be
  \[\theta_{e_t}^{\ell}=\begin{cases}
      \mathrm{id}                            & \text{(for $1-r\le \ell < d(e_t)-r$)}, \\
      \pathtohom{[\gamma_{e_t}^1:\deg(e_t)]} & \text{(otherwise)},
    \end{cases}\]
  The map $F(\theta_{e_t}): F(T_{e_t})\to F(T_{e_t})$ is an irreducible map over $B$-module from $Q_{e_t}$ to $Q_{e_t}$.
  Moreover, the union set of $F(\tau_n)$ and $F(\theta_n)$ where $n$ runs over edges in $E$ is the set of all irreducible maps over $B$-modules up to the scalar multiplication.
  It holds that
  \begin{gather*}
    F(\theta_{e_t})^{m_{e_t}+1}=0, \\
    \bigl(F(\tau_{e_t})\circ F(\tau_{e_{t+1}})\circ\dots\circ F(\tau_{e_{t-1}})\bigr)^{m_{v_0}}\circ F(\tau_{e_t})=0, \\F(\theta_{e_t})^{m_{e_t}}=\bigl(F(\tau_{e_t})\circ F(\tau_{e_{t+1}})\circ\dots\circ F(\tau_{e_{t-1}})\bigr)^{m_{v_0}},\\
    F(\theta_{e_t})\circ F(\tau_{e_t})=0,\\
    F(\tau_{e_t})\circ F(\theta_{e_{t+1}})=0
  \end{gather*}
  for all $t\in \{1,\dots,N\}$.
  This implicates that $B$ is a generalized Brauer tree algebra associated to a generalized Brauer star, identified with a rooted generalized Brauer tree having the same preorder traversal of multiplicities as $\mathcal T$. Here, we call generalized Brauer trees having a root such that the depth of the all vertices except for the root are equal to $1$ generalized Brauer stars.
\end{remark}

We define some modules to describe the images of simple modules through the functor $F^{-1}$ in the following proposition. For $S_n \in \mathcal S$, let $L_n$ denote an $A$-module \[\langle \{\,[\gamma_{n}^{1}:j] + I \,\mid\, j\in\{(\deg(n)-1)m_{n}+1,\dots,\deg(n)m_{n}\}\}\,\rangle_k \]
\[\cong\,
  \fbox{$\begin{matrix}
        {c(n,1)} \\ {c(n,2)}\\\vdots\\{c(n,\deg(n)-2)}\\{c(n,\deg(n)-1)}\\{n}
      \end{matrix}$}\,.\]
We note that these modules are not influenced by multiplicities $m$.
Let $\iota_t:\syzygy[-t]L_n \to I(\syzygy[-t]L_n)$ denote an injective hull of $\syzygy[-t]L_n$ for $t\in \mathbb Z$.
We remark that $I(\syzygy[-1] L_n)$ is isomorphic to $P_{c(p_n, s(n)-1)}$ if $m_n=1$ and $P_{c(p_n, s(n)-1)}\oplus P_n$ if $m_n\ge 2$.

\begin{proposition}\label{simpleminded} For $S_n\in \mathcal{S}_i - \mathcal{S}_{i-1}$, it holds that $F^{-1}(V_n)$ is isomorphic to
  $L_n[i]$ in $\derb(A)$.
\end{proposition}

\begin{proof}
  We show that $L_n[i]$ is isomorphic to the complex $Y_n=(Y_n^{\ell}, d_n^{\ell})_{\ell \in \mathbb Z}$ constructed in the same manner as \cref{simplemindeddecreasing}.
  In this proof, we set $\mathcal S_{r+\ell}=\mathcal S$ and $\mathcal S_{-2-\ell}=\emptyset$ for all $\ell \in \mathbb Z_{\ge 0}$ as a matter of convenience.
  We remark that the composition factors of $P_n$ are in $\mathcal S_{i+1}$ but not in $\mathcal S_{i-2}$ by \cref{indecprojstr}.

  If $i=0$, we have $Y_n=Y_n^0=S_n=L_n$. Therefore the proposition is satisfied for $i=0$.

  If $i\neq 0$,
  then $Y_n^{-i} = P_n$.
  A quotient module $L_n/S_n$ is isomorphic to a uniserial module
  \[ \fbox{$\begin{matrix}
          {c(n,1)} \\ {c(n,2)}\\\vdots\\{c(n,\deg(n)-2)}\\{c(n, \deg(n)-1)}
        \end{matrix}$}.\,\]
  If $m_n=1$, then the socle of $P_n/L_n$ is isomorphic to $S_{c(p_n, s(n)-1)}$, which is not in $\mathcal S_{i-1}$.
  If $m_n\ge 2$, then the socle of $P_n/L_n$ is isomorphic to a direct sum of $S_{c(p_n, s(n)-1)}$ and $S_{n}$, both of which are not in $\mathcal S_{i-1}$.
  Therefore, $L_n$ is the largest submodule of $P_n$ such that the all composition factors of $L_n/S_n$ are in $\mathcal S_{i-1}$. We have $Y_n^{-i+1} = I(\syzygy[-1]L_n)$ and $d_n^{-i}$ is $\coker\iota_0:P_n \to P_n/L_n$ if $i=1$, otherwise $d_n^{-i}$ is the composition map of $\coker\iota_0:P_n \to P_n/L_n$ with an injective hull $P_n/L_n \to I(\syzygy[-1] L_n)$.
  Consider the following lemma.
  \begin{lemma}\label{essential}
    For all $0 \le t< i-1$, it holds that \begin{itemize}
      \item $Y_n^{-i+t}=I(\syzygy[-t]L_n)$, $Y_n^{-i+t+1} =I(\syzygy[-t-1] L_n),$
      \item the composition factors of $Y_n^{-i+t+1}$ are in $\mathcal S_{i+t+2}$ but not in $\mathcal S_{i-(t+2)}$, and
      \item $d_n^{-i+t}$ is the composition map of $\coker\iota_t: I(\syzygy[-t] L_n)\to \syzygy[-t-1]L_n$ with $\iota_{t+1}$.
    \end{itemize}
    For $t=i-1$, we have $Y_n^{0} =\syzygy[-i] L_n$ and $d_n^{-1}=\coker\iota_{i-1}: I(\syzygy[-i+1] L_n)\to \syzygy[-i]L_n.$

  \end{lemma}
  \begin{proof}[Proof of \cref{essential}]
    We prove the lemma by the induction on $t\in \{0,\dots, i\}$.
    Consider the case for $t=0$.
    In this case, we have $Y_n^{-i}=P_n= I(L_n)$.
    If $i=1$, we conclude that the lemma is valid since $Y_n^{0}=L_n/S_n=\syzygy[-1]L_n$.
    If $i\ge 2$, then $Y_n^{-i+1}=I(\syzygy[-1]L_n)$. Hence, the composition factors of $Y_n^{-i+1}$ are in $\mathcal S_{i+2}$ but not in $\mathcal S_{i-2}$.

    We prove that the lemma is true for $0<t< i-1$.
    We have $\coker d_n^{-i+t-1}=\coker\iota_t: Y_n^{-i+t}\to \syzygy[-t-1] L_n$, since $d_n^{-i+t-1}=\iota_t \circ \coker \iota_{t-1}$.
    The composition factors of $\syzygy[-t-1] L_n=\image(\coker d_n^{-i+t-1})$ are in the set of the composition factors of $Y_n^{-i+t}$, which is contained in $\mathcal S_{i+t+1}$ but not in $\mathcal S_{i-(t+1)}$ by assumption. Consequently, $\syzygy[-t-1] L_n$ has no proper submodules by which the composition factors of the quotient module of $\syzygy[-t-1] L_n$ are in $\mathcal S_{i-(t+1)}$.
    Hence, $d_n^{-i+t}$ is the composition map of $\coker \iota_t$ with an injective hull $\iota_{t+1}$.
    For a contradiction, assume that $Y_n^{-i+t}=I(\syzygy[-t] L_n)$ has an indecomposable summand $Q$
    such that the composition of $d_n^{-i+t}$ with $\pi$ the projection on $Q$ is $0$. Then $\pi\circ \iota_{t+1}=0$ since $\coker \iota_t$ is an epimorphism. This contradicts our assumption that $\iota_{t+1}$ is an injective hull of $\syzygy[-t-1] L_n$.
    Thus, each indecomposable summands of $Y_n^{-i+t}$ has a nontrivial image through $d_n^{-i+t-1}$. Since the composition factors of $Y_n^{-i+t}$ are in $\mathcal S_{i+t+1}$ but not in $\mathcal S_{i-(t+1)}$, each indecomposable summands of $Y_n^{-i+t}$ is isomorphic to $P_{\ell}$ such that $S_\ell\in \mathcal S_{i+t}-\mathcal S_{i-t}$. By \cref{dimHom}, each indecomposable summands of $Y_n^{-i+t+1}$ is isomorphic to $P_{\ell}$ such that $S_\ell\in \mathcal S_{i+t+1}-\mathcal S_{i-t-1}$. This indicates that the composition factors of $Y_n^{-i+t+1}$ are in $\mathcal S_{i+t+2}$ but not in $\mathcal S_{i-(t+2)}$.

    Similarly, we have $Y_n^0=\syzygy[-i]L_n$ and $d_n^{-1}=\coker \iota_{i-1}$ by the fact that the composition factors of $Y_n^{-1}=\syzygy[-i+1]L_n$ are not in $\mathcal S_{0}$ and by the definition of construction of $d_n^{-1}$ and $Y_n^0$ stated as in \cref{s:perverse}.
  \end{proof}
  \cref{essential} asserts that $Y_n$ is a shifted truncated minimal injective resolution of $L_n$ having $P_n$ in $-i$th term and $\syzygy[-i] L_n$ in $0$th term. Let $f$ be a homomorphism of complexes from $L_n[i]$ to $Y_n$ satisfying $f^{-i}=\iota_0$
  and $f^{\ell}=0$ for all $\ell \in \mathbb Z$ except for $-i$. Since the mapping cone of $f$ is an exact sequence, the homomorphism $f$ is a quasi-isomorphism. Therefore, we have $Y_n \cong L_n[i]$ in $\derb{(A)}$.
\end{proof}

Let $M$ be an indecomposable $(B,A)$-bimodule inducing a stable equivalence of Morita type between $B$ and $A$ making the following diagram commutative:
\[\begin{tikzcd}[]
    \derb(B)\dar[""]\rar["F^{-1}"]&\derb(A)\dar[""]
    \\
    \derb(B)/\homob(\projcat B) \dar[phantom,"\rotatebox{-90}{$\cong$}"] & \derb(A)/\homob(\projcat A) \dar[phantom,"\rotatebox{-90}{$\cong$}"]\\
    \stmodcat{B}\rar["-\otimes_{B} M"]&\stmodcat{A}
  \end{tikzcd} \]
(see \cref{derivedandstable}).
It follows that $V_n\otimes_B M$ is isomorphic to $\syzygy[d(n)-r] L_n$ in $\stmodcat A$ for $n\in E$, since all terms of $Y_n$ are projective except for the $0$th term $\syzygy[d(n)-r] L_n$. By \cite[Theorem 2.1 (ii)]{Linckelmann1996}, $V_n\otimes_B M$ is isomorphic to $\syzygy[d(n)-r] L_n$ in $\modcat A$ for $n\in E$.
Let $P^{\bullet}=(P^t, d_M^t)_{t\in \Z}$ be a minimal projective resolution of the $B^{\op}\otimes_k A$-module $M$. By \cref{prop:hellermove}, we have \[V_n\otimes_B \syzygy[t-1] M \cong \syzygy[t-1](V_n\otimes_B M)\cong \syzygy[t-1+d(n)-r] L_n\] for all $n\in E$ and $t>0$. By \cref{projresol}, we have the following proposition.
\begin{proposition}\label{prop:projresol}
  The $(B,A)$-bimodule $P^{-t}$ is isomorphic to
  \[
    \begin{cases}
      M                                                                   & (t=0),   \\
      \bigoplus_{n\in E} Q^{\ast}_{n}\otimes_k P(\syzygy[t-1+d(n)-r] L_n) & (t > 0), \\
      0                                                                   & (t < 0)
    \end{cases}\] for each integer $t$.
\end{proposition}
Let $C=(C^t, d_C^t)_{t\in \Z}$ be a subcomplex of $P^{\bullet}$ defined as follows:
\[C^{-t}=\begin{cases}
    M                                                & (t=0),   \\
    \bigoplus_{n\in E,~~d(n)\le {r-t}}
    Q^{\ast}_{n}\otimes_k P(\syzygy[t-1+d(n)-r] L_n) & (t > 0), \\
    0                                                & (t < 0),
  \end{cases}\] and the map $d_C^t$ is the restriction $d_M^t$ to $C^{-t}$ for $t\in\mathbb Z$. This complex is bounded, since $C^{-t}=0$ for $t< 0$ and $t\ge r$. We should confirm that $d^t_C$ is well-defined, which we prove after the following lemma.

\begin{lemma}\label{essentialproj}
  The composition factors of $P(\Omega^u\, L_{\ell})$ are in $\mathcal S_{r-d(\ell)+u}$ for all $0\le u <d(\ell)$.
\end{lemma}
\begin{proof}
  Let $\pi_u:P(\Omega^u\, L_{\ell})\to \Omega^u\, L_{\ell}$ be a projective cover of $\Omega^u\,L_{\ell}$.
  We give a proof by the induction on $u$, which is a dual proof of a part of \cref{essential}. The lemma is valid for $u=0$, since the composition factors of $P(L_{\ell})=P_{c(\ell,1)}$ are in $\mathcal S_{r-d(\ell)}$. We show that the lemma is true for $u>0$ under the assumption that the proposition is true for $u-1$.
  For contradiction, assume that $P(\syzygy[u] L_{\ell})$ has an indecomposable summand $Q$
  such that the composition of $(\ker \pi_{u-1})\circ \pi_u : P(\Omega^u\,L_{\ell})\to P(\syzygy[u-1]L_{\ell}) $ with $\iota$ the injection from $Q$ is $0$. Then $\pi_u\circ \iota=0$ since $\ker \pi_{u-1}$ is a monomorphism. This contradicts our assumption that $\pi_u$ is a projective cover of $\syzygy[u] L_{\ell}$.
  Thus, each indecomposable summands of $P(\Omega^u\,L_{\ell})$ has a nontrivial image on $P(\syzygy[u-1]L_{\ell})$ by $(\ker \pi_{u-1})\circ \pi_u$.

  Since the composition factors of $P(\syzygy[u-1] L_{\ell})$ are in $\mathcal S_{r-d(\ell)+u-1}$ by the assumption of induction, each indecomposable summands of $P(\syzygy[u-1] L_{\ell})$ is isomorphic to $P_{y}$ such that $S_y\in \mathcal S_{r-d(\ell)+u-2}$. By \cref{dimHom}, each indecomposable summands of $P(\Omega^u\, L_{\ell})$ is isomorphic to $P_{y}$ such that $S_y\in \mathcal S_{r-d(\ell)+u-1}$. This indicates that the composition factors of $P(\Omega^u\, L_{\ell})$ are in $S_{r-d(\ell)+u}$.
\end{proof}

\begin{proposition}
  $d_C^t$ is well-defined for $t\in \mathbb Z$.
\end{proposition}
\begin{proof}
  Put
  \[D^{-t}=\begin{cases}
      \bigoplus_{n\in E, ~~ d(n)> {r-t}} Q^{\ast}_{n}\otimes_k P(\syzygy[t-1+d(n)-r] L_n) & (t>0), \\ 0 & (t \le 0),
    \end{cases}\]
  for all $t\in\mathbb Z$. Since $C^{-t}\oplus D^{-t}\cong P^{-t}$, it is enough to show that $\Hom_{B^{\op}\otimes_k A}(C^{-t}, D^{-t+1})=0$ for all $r < t< 1$.
  Take $t,\ell \in \mathbb Z$ and $n\in E$ to satisfy that $r<t<1$, $t+d(n)-r\le 0$ and $t-1+d(\ell) -r >0$.
  By \cref{essential}, if $t+d(n)-r\neq 0$, the composition factors of $I(\syzygy[t+d(n)-r] L_n)$ are not in $\mathcal S_{t-1}$. If $t+d(n)-r= 0$, then the composition factors of $I(\syzygy[t+d(n)-r] L_n)=P_n$ are not in $\mathcal S_{r-d(n)-2}=\mathcal S_{t-2}$.
  As a result, the composition factors of $I(\syzygy[t+d(n)-r] L_n)$ are not in $\mathcal S_{t-2}$. By \cref{essentialproj}, the composition factors of $P(\syzygy[t-2+d(\ell)-r] L_{\ell})$ are in $\mathcal S_{t-2}$.
  As a result, we have $\Hom_A(I(\syzygy[t+d(n)-r]L_n), P(\syzygy[t-2+d(\ell)-r]L_{\ell}))=0$. Since $P(\syzygy[t-1+d(n)-r]L_n) \cong I(\syzygy[t+d(n)-r] L_n)$, it follows
  \begin{gather*}
    C^{-t}\cong \bigoplus_{n\in \{1,\dots,N\},~~d(n)\le r-t} I(\syzygy[t+d(n)-r] L_n)^{\oplus \dim(Q_n)},\\
    D^{-t+1} \cong \bigoplus_{\ell\in \{1,\dots,N\},~~d(\ell)> r-t+1} P(\syzygy[t-2+d(\ell)-r] L_{\ell})^{\oplus \dim(Q_{\ell})}
  \end{gather*} as $A$-modules. Therefore, we have \[\Hom_{B^{\op}\otimes_k A}(C^{-t}, D^{-t+1})\subseteq \Hom_{A}(C^{-t}, D^{-t+1})=0.\]
\end{proof}

\begin{proposition}\label{where simple moves on C}
  For all $n\in E$,
  \[V_n\otimes_B C\cong L_n[r-d(n)]\] in $\derb(A)$.
\end{proposition}
\begin{proof}
  We remark that \[V_n \otimes_B Q_{{\ell}}^{\ast}\cong \Hom_B(Q_{\ell}, V_n)=\delta_{\ell n}k\] for all $\ell,n\in\{1,\dots, N\}$ since $Q_{\ell}$ is an indecomposable projective module. Consequently, the $-t$th term of the complex $V_n\otimes_B C$ is
  \[\begin{cases}
      \syzygy[d(n)-r]L_n         & (t=0),         \\
      P(\syzygy[t-1+d(n)-r] L_n) & (d(n)\le r-t), \\ 0 & (d(n)> r-t).
    \end{cases}\]
  The complex $V_n\otimes_B C$ is a truncated projective resolution of $\syzygy[d(n)-r] L_n$ having $P(\syzygy[-1] L_n)\cong I(L_n)$ in $(d(n)-r)$th term and having $0$ in $\ell$th term for $\ell < d(n)-r$ and $0 < \ell$.
  Let $f$ be a homomorphism of complexes from $L_n[r-d(n)]$ to $V_n\otimes_B C$ satisfying $f^{d(n)-r}=\iota_0$
  and $f^{\ell}=0$ for all $\ell \in \mathbb Z$ except for $d(n)-r$. Since the mapping cone of $f$ is an exact sequence, the homomorphism $f$ is a quasi-isomorphism. Therefore, we have $V_n \otimes_B C \cong L_n[r-d(n)]$ in $\derb{(A)}$.
\end{proof}

\begin{theorem}\label{Cistwo-sided}
  The complex $C$ is a two-sided tilting complex of $B^{\op}\otimes_k A$-modules.
\end{theorem}
\begin{proof}
  We note that $C^t$ is projective for all $t < 0$ and $C^0=M$.
  By \cref{prop:two-sided}, it is enough to show \[\Hom_{D^b(A)}(V_n\otimes_B C, V_{\ell}\otimes_B C[-u])\cong \delta_{n\ell}\delta_{u 0}k\]
  for $n,\ell \in E$ and $u\in \mathbb Z_{\ge 0}$. By \cref{where simple moves on C,simpleminded}, we have
  \begin{gather*}
    V_n\otimes_B C\cong L_n[r-d(n)]\cong F^{-1}(V_n),\\
    V_{\ell}\otimes_B C[-u]\cong L_n[r-d(\ell)][-u]\cong F^{-1}(V_{\ell})[-u].
  \end{gather*}
  Applying these to \cref{treetostarcomplexisdecreasingperverse}, we have
  \begin{align*}
    \Hom_{D^b(A)}(V_n\otimes_B C, V_{\ell}\otimes_B C[-u]) & \cong \Hom_{D^b(A)}(F^{-1}(V_n), F^{-1}(V_{\ell})[-u]) \\
                                                           & \cong \Hom_{D^b(B)}(V_n, V_{\ell}[-u])                 \\
                                                           & \cong \delta_{n\ell}\delta_{u0}k.
  \end{align*}
\end{proof}

Since the restriction of $C$ is projective on either side, by \cref{Cistwo-sided,omittingL} we conclude that $G:={-}\otimes_B C$ is a triangulated functor from $D^b(B)$ to $D^b(A)$.
\begin{proposition}\label{perverseG}
  The equivalence $G$ is perverse relative to $(\mathcal{S}_{\bullet}', \mathcal{S}_{\bullet}, -q)$, where $q:\{0,\dots,r-1\}\to \mathbb Z$ is a map given by $q(i)=-i$.
\end{proposition}
\begin{proof}
  Since $F$ is perverse relative to $(\mathcal{S}_{\bullet},\mathcal{S}_{\bullet}', q)$ by \cref{decreasing,treetostarcomplexisdecreasingperverse}, $F^{-1}$ is perverse relative to $(\mathcal{S}_{\bullet}',\mathcal{S}_{\bullet}, -q)$ by \crefenum{usefulperverse}{usefulperverse1}. As we have $G(V) \cong F^{-1}(V)$ for all $V\in \mathcal S'$ by \cref{simpleminded,where simple moves on C}, the definition of perverse equivalence stated as in \cref{def:perverse} implies that the functor $G$ is perverse relative to the same filtrations and map as the functor $F^{-1}$ is.
\end{proof}
\begin{theorem}\label{thm:resisC}
  The restriction of the complex $C$ of $(B,A)$-bimodules to $A$-modules is isomorphic to the tilting complex $T$ in $D^b(A)$.
\end{theorem}
\begin{proof}
  Since $F\circ G$ is perverse relative to $(\mathcal{S}_{\bullet}', \mathcal{S}_{\bullet}', 0)$ by Propositions \ref{perverseG} and \crefenum{usefulperverse}{usefulperverse2}, the functor $F\circ G$ restricts to a Morita equivalence from $B$ to $B$ by \crefenum{usefulperverse}{usefulperverse3}. Therefore, we have \[F(B\otimes_B C)\cong F\circ G (B) \cong B\cong F(T).\]
  By applying $F^{-1}$, we have $B\otimes_B C \cong T$. This concludes that the restriction of $C$ to $A$ is isomorphic to $T$.
\end{proof}
\begin{remark}
  The functor $J:=\Hom_A(C,{-})$ is equivalent to ${-}\otimes_A C^{\ast}:\derb(A)\to\derb(B)$ since $A$ is a symmetric algebra. As the adjoint functor of the equivalent functor $G={-}\otimes_B C$ is the inverse functor, we have
  \[J\cong G^{-1}\cong F\] as triangulated functors.
\end{remark}

\section{Example}\label{s:examples}
We use the hypotheses and notation made in \cref{s:main}. We give an example of the theorems in \cref{s:main}. Let $(\mathcal T,v_0)$ be a generalized Brauer tree algebra with a root $v_0$ stated in \cref{myexample}. Here, we denote the structure of indecomposable projective modules $P_n$ \[P_n\cong \begin{bmatrix}
    S_n                  \\
    U_n^{+}\quad U_n^{-} \\
    S_n
  \end{bmatrix}\] to describe $P_n/\rad{P_n}\cong S_n$, $\soc P_n \cong S_n$ and $\rad{P_n}/\soc{P_n}\cong U_n^+\oplus U_n^{-}$. Following to \cref{indecprojstr}, it holds that
\begin{align*}
  P_1=\begin{bmatrix}
    1                          \\
    \begin{matrix}
      6
    \end{matrix}
    \quad
    \begin{matrix}
      2 \\3\\5
    \end{matrix} \\
    1
  \end{bmatrix}, \quad
  P_2=\begin{bmatrix}
    2 \\
    3 \\5\\1 \\
    2
  \end{bmatrix}, \quad
  P_3=\begin{bmatrix}
    3                          \\
    \begin{matrix}
      5 \\1\\2
    \end{matrix}
    \quad
    \begin{matrix}
      4 \\3\\4
    \end{matrix} \\
    3
  \end{bmatrix},\quad
  P_4=\begin{bmatrix}
    4 \\
    3 \\
    4 \\
    3 \\
    4
  \end{bmatrix},\quad
  P_5=\begin{bmatrix}
    5                          \\
    \begin{matrix}
      1 \\
      2 \\
      3
    \end{matrix}
    \quad
    \begin{matrix}
      5 \\5
    \end{matrix} \\
    5
  \end{bmatrix},\quad
  P_6=\begin{bmatrix}
    6                          \\
    \begin{matrix}
      1
    \end{matrix}
    \quad
    \begin{matrix}
      6
    \end{matrix} \\
    6
  \end{bmatrix}.
\end{align*}
We have a tree-to-star complex $T=\bigoplus_{n=1}^6 T_n$ of $A$-modules made in \cite{MEMBRILLOHERNANDEZ1997231} satisfying as follows:
\[\begin{tikzcd}[column sep=huge]
    \phantom{0}&-2\text{nd}&-1\text{st}& 0\text{th}\\[-10pt]
    T_1\rar[":",phantom]&[-35pt]P_1\rar[]& 0\rar[]& 0\\
    T_2\rar[":",phantom]&P_1\rar["\pathtohom{(\gamma_1^2\gamma_1^3\gamma_1^4)}"]&P_2\rar[]& 0\\
    T_3\rar[":",phantom]&P_1\rar["\pathtohom{(\gamma_1^3\gamma_1^4)}"]&P_3\rar[]&0\\
    T_4\rar[":",phantom]&P_1\rar["\pathtohom{(\gamma_1^3\gamma_1^4)}"]&P_3\rar["\pathtohom{\gamma_3^2}"]&P_4\\
    T_5\rar[":",phantom]&P_1\rar["\pathtohom{\gamma_1^4}"]&P_5\rar[]& 0\\
    T_6\rar[":",phantom]&P_6\rar[]&0\rar[]& 0
  \end{tikzcd}\]
According to \cref{rem:endomorphism}, the algebra $\End_{\homob{(\projcat{A})}}(T)$ derived equivalent to $A$ is a generalized Brauer tree algebra associated to a generalized Brauer star as follows:
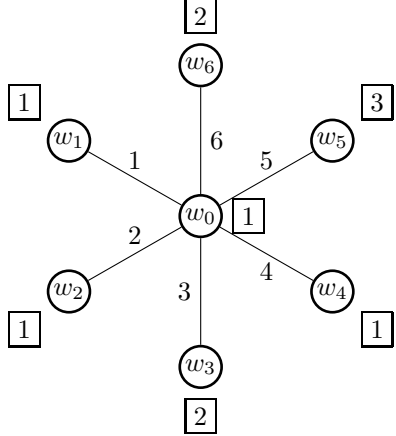
\[\begin{tikzpicture}
    \tikzset{block/.style={circle,draw,thick,minimum width=0, line width=1,minimum height=0,inner sep=1.0}};
    \node (P0) [label={[label distance=.1]0:\fbox{1}},block]at (0:0) {$w_0$};
    \foreach \x / \y[count=\i] in {1/1, 2/1, 3/2, 4/1, 5/3, 6/2}
      {
        \node (P\i) [label={[label distance=.1]{90+60*\x}:\fbox{\y}},block] at (90+60*\x:2) {$w_\i$};
      }
    \path[]
    (P0) edge node[above] {$1$} (P1)
    (P0) edge node[above] {$2$} (P2)
    (P0) edge node[left] {$3$} (P3)
    (P0) edge node[below] {$4$} (P4)
    (P0) edge node[above] {$5$} (P5)
    (P0) edge node[right] {$6$} (P6);
  \end{tikzpicture}\]
We note for $(\mathcal T,v_0)$, \[d(1)=d(6)=1,\quad d(2)=d(3)=d(5)=2,\quad d(4)=3.\]
By \cref{treetostarcomplexisdecreasingperverse,perverseandtreetostar}, an equivalence $F=\Hom_A^{\bullet}(T,{-}):\derb(A)\to $\\$ \derb(\End_{\homob(\projcat A)}(T))$
  induces the decreasing perverse with respect to the filtration of a complete set of representatives of isomorphism classes of simple $A$-modules \[ \emptyset \subseteq \{S_4\}\subseteq \{S_2,S_3, S_4, S_5\} \subseteq \{S_1,S_2,S_3, S_4, S_5,S_6\}.\]
  We have uniserial $A$-modules as follows:
  \begin{align*}
    L_1\cong
    \myfbox{$\begin{matrix}
          2 \\3\\5\\1
        \end{matrix}$},\quad
    L_2\cong S_2,\quad
    L_3\cong \myfbox{$\begin{matrix}
          4 \\3
        \end{matrix}$},\quad
    L_4\cong S_4,\quad
    L_5\cong S_5,\quad
    L_6\cong S_6.
  \end{align*}
  \cref{simpleminded} asserts that
  \begin{align*}
    F^{-1}(V_1)\cong
    \myfbox{$\begin{matrix}
          2 \\3\\5\\1
        \end{matrix}$}[2],\quad
    F^{-1}(V_2)\cong S_2[1],\quad
    F^{-1}(V_3)\cong\myfbox{$\begin{matrix}
          4 \\3
        \end{matrix}$}[1], \\
    F^{-1}(V_4)\cong S_4,\quad
    F^{-1}(V_5)\cong S_5[1],\quad
    F^{-1}(V_6)\cong S_6[2]
  \end{align*}
  in the bounded derived category of $A$-modules.
  Let $M$ be an $\End(T)^{\op}\otimes A$-module giving a stable equivalence of Morita type induced by $T$.
  By \cref{prop:projresol}, a minimal projective resolution of $M$ is described as follows:
  \[
    \begin{tikzcd}[row sep=tiny]
      \cdots &-2\text{nd}&-1\text{st}& 0\text{th}\\
      \cdots & {Q^{\ast}_{1}\otimes_k P(\syzygy[-1] L_1)}& Q^{\ast}_{1}\otimes_k P(\syzygy[-2] L_1)& \phantom{0}\\
      &\oplus & \oplus &\\
      \cdots& Q^{\ast}_{2}\otimes_k P( L_2)& Q^{\ast}_{2}\otimes_k P(\syzygy[-1] L_2)& \phantom{0}\\
      &\oplus & \oplus &\\
      \cdots& Q^{\ast}_{3}\otimes_k P( L_3)& Q^{\ast}_{3}\otimes_k P(\syzygy[-1] L_3)\\
      \rar[shorten >=40pt]&\oplus\rar[shorten=35pt] & \oplus\rar[shorten <= 40pt,] &M\\
      \cdots& Q^{\ast}_{4}\otimes_k P(\syzygy[] L_4)& Q^{\ast}_{4}\otimes_k P( L_4)& \phantom{0}\\
      &\oplus & \oplus &\\
      \cdots& Q^{\ast}_{5}\otimes_k P( L_5)& Q^{\ast}_{5}\otimes_k P(\syzygy[-1] L_5)& \phantom{0}\\
      &\oplus & \oplus &\\
      \cdots& Q^{\ast}_{6}\otimes_k P(\syzygy[-1] L_6)& Q^{\ast}_{6}\otimes_k P(\syzygy[-2] L_6)& \phantom{0}\\
    \end{tikzcd}
  \]
  Our two-sided tilting complex $C$ is given as follows:
\[
  \begin{tikzcd}[row sep=tiny]
    \phantom{\cdots} &-2\text{nd}&-1\text{st}& 0\text{th}\\
    \phantom{\cdots} & {Q^{\ast}_{1}\otimes_k P(\syzygy[-1] L_1)}& Q^{\ast}_{1}\otimes_k P(\syzygy[-2] L_1)& \phantom{0}\\
    &\oplus & \oplus &\\
    \phantom{\cdots}& Q^{\ast}_{2}\otimes_k P( L_2)& Q^{\ast}_{2}\otimes_k P(\syzygy[-1] L_2)& \phantom{0}\\
    &\oplus & \oplus &\\
    \phantom{\cdots}& \phantom{Q^{\ast}_{3}\otimes_k P( L_3)}& Q^{\ast}_{3}\otimes_k P(\syzygy[-1] L_3)\\
    \rar[shorten >=40pt,phantom]&\phantom{\oplus}\rar[shorten=35pt] & \oplus\rar[shorten <= 40pt,] &M\\
    \phantom{\cdots}& \phantom{Q^{\ast}_{4}\otimes_k P(\syzygy[] L_4)}& \phantom{Q^{\ast}_{4}\otimes_k P( L_4)}& \phantom{0}\\
    &\phantom{\oplus} & \phantom{\oplus} &\\
    \phantom{\cdots}& \phantom{Q^{\ast}_{5}\otimes_k P( L_5)}& Q^{\ast}_{5}\otimes_k P(\syzygy[-1] L_5)& \phantom{0}\\
    &\phantom{\oplus} & \oplus &\\
    \phantom{\cdots}& Q^{\ast}_{6}\otimes_k P(\syzygy[-1] L_6)& Q^{\ast}_{6}\otimes_k P(\syzygy[-2] L_6)& \phantom{0}\\
  \end{tikzcd}
\]

\section*{Acknowledgements}
The authors would like to thank Professor Naoko Kunugi for her advice and help.

\printbibliography

\vspace{\baselineskip}
\par\noindent
Shuji Fujino~~1124702@ed.tus.ac.jp\par\noindent
Department of Mathematics, Graduate School of Science, Tokyo University of Science, 1-3 Kagurazaka, Shinjuku-ku, Tokyo, 162-8601, Japan
\vspace{\baselineskip}
\par\noindent
Yuta Kozakai~~kozakai@rs.tus.ac.jp\par\noindent
Department of Mathematics, Tokyo University of Science, 1-3 Kagurazaka, Shinjuku-ku, Tokyo 162-8601, Japan
\vspace{\baselineskip}
\par\noindent
Kohei Takamura~~1122514@alumni.tus.ac.jp\par\noindent
Department of Mathematics, Graduate School of Science, Tokyo University of Science, 1-3 Kagurazaka, Shinjuku-ku, Tokyo, 162-8601, Japan

\end{document}